\documentclass[a4ppaper,14pt]{amsart}

\usepackage{enumerate}
\usepackage[latin1]{inputenc}
\usepackage{amsmath, dsfont}
\usepackage{amsthm}
\usepackage{latexsym}
\usepackage{amssymb}
\usepackage{amsmath}
\usepackage{comment}
\usepackage[all]{xy}
\usepackage{calc}
\usepackage{multicol}
\usepackage{slashed}
\usepackage{graphicx}
\usepackage{subfigure}

\newtheorem{thm}{Theorem}[section]
\newtheorem{prop}{Proposition}[section]
\newtheorem{defi}{Definition}[section]
\newtheorem{lem}{Lemma}[section]
\newtheorem{cor}{Corollary}[section]

\newtheorem{rem}{Remark}[section]
\theoremstyle{notation}

\newcommand{\R}{\mathbb{R}}
\numberwithin{equation}{section}

\newcommand{\eps}{\varepsilon}

\newcommand{\vertiii}[1]{{\left\vert\kern-0.25ex\left\vert\kern-0.25ex\left\vert #1
\right\vert\kern-0.25ex\right\vert\kern-0.25ex\right\vert}}

\usepackage[textwidth=138mm,
textheight=215mm,
left=30mm,
right=30mm,
top=25.4mm,
bottom=25.4mm,
headheight=2.17cm,
headsep=4mm,
footskip=12mm,
heightrounded,]{geometry}

\usepackage{xcolor}

\makeatletter
\newcommand{\leqnomode}{\tagsleft@true}
\newcommand{\reqnomode}{\tagsleft@false}
\makeatother

\title{Scattering and blow-up for Chern-Simons-Schr\"odinger equations in the mass supercritical case}

\author[Gou and Zhang]{Vladimir Georgiev \and Tianxiang Gou}

\address{Vladimir Georgiev
\newline \indent Dipartimento di Matematica, Universit\'a di Pisa,
\newline \indent Largo B. Pontecorvo 5, 56100 Pisa, Italy.
\newline \indent Faculty of Science and Engineering, Waseda University,
\newline \indent 3-4-1, Okubo, Shinjuku-ku, Tokyo 169-8555, Japan.
\newline \indent Institute of Mathematics and Informatics, Bulgarian Academy of Sciences, Acad. Georgi Bonchev Str., Block 8, 1113 Sofia, Bulgaria.}
\email{georgiev@dm.unipi.it}

\address{Tianxiang Gou
\newline \indent School of Mathematics and Statistics, Xi'an Jiaotong University,
\newline \indent Xi'an, Shaanxi 710049, People's Republic of China.}
\email{tianxiang.gou@xjtu.edu.cn}
\begin{document}

\begin{abstract}
In this paper, we are concerned with solutions to the Cauchy problem for Chern-Simons-Schr\"odinger equations in the mass supercritical case. First we establish the local well-posedness of solutions in the radial space.  Then we consider scattering versus blow-up dichotomy for radial data below the ground state threshold.
\medskip



\end{abstract}

\maketitle

\section{Introduction}

\reqnomode

In this paper, we are interested in solutions to the Cauchy problem for the following Chern-Simons-Schr\"odinger (CSS) equation in two spatial dimensions,
\begin{align} \label{CSS}
\left\{
\begin{aligned} 
\textnormal{i} \, D_t u+ \left(D_1 D_1 +D_2 D_2\right) u &= - |u|^{p-1} u,\\
\partial_t A_1-\partial_1 A_0&=-\mbox{Im}\left(\overline{u} \,D_2 u\right), \\
\partial_t A_2-\partial_2 A_0&=\mbox{Im}\left(\overline{u} \,D_1 u\right), \\
\partial_1 A_2-\partial_2 A_1&=- 1/2 \,|u|^2,
\end{aligned}
\right.
\end{align}
under the Coulomb gauge condition
\begin{align} \label{gauge}
\partial_1 A_1 + \partial_2 A_2 =0,
\end{align}
where $\textnormal{i}$ denotes the imaginary unit, $p>3$ and
$D_t, D_{i}$ are covariant derivative operators defined by
$$
D_{t}=\partial_{t} + \textnormal{i} \,  A_{0}, \quad D_{i}=\partial_{i} + \textnormal{i} \,  A_{i}  \quad \mbox{for} \,\,\, i=1, 2.
$$
The CSS equation under consideration is a gauge-covariant nonlinear Schr\"odinger equation with a long-range electromagnetic field introduced by Jackiw and Pi in \cite{JaPi1, JaPi2}.  It serves as a basic model for Chern-Simons dynamics on the plane used to describe a number of important planar phenomena such as anyonic statistics, fractional quantum Hall effect and high $T_c$ superconductors, see for example \cite{Du, EHA, EHA1, HoZh, Ya} and references therein.


Here we shall consider solutions to the Cauchy problem for \eqref{CSS} in the radial symmetric context. We suppose that $u(t, x)$ is radially symmetric and the electromagnetic potentials satisfy the following radially symmetric ansatz,
$$
A_0(t,x) = A_0(t,|x|),  \quad A_1(t,x) = -\frac{x_2}{|x|} h(t,|x|), \quad A_2(t,x) = \frac{x_1}{|x|} h(t,|x|),
$$
where $h$ is assumed to be radially symmetric.

The first aim of the present paper is to study the local well-posedness of solutions to the Cauchy problem for \eqref{CSS}. Let us now briefly review some related results in this direction. When $p=3$, the Cauchy problem for \eqref{CSS} was initially investigated in \cite{BBS}, where the authors established the local well-posedness of solutions in $H^2(\R^2)$ and also proved the existence of global solutions in $H^1(\R^2)$ for initial data with small $L^2$-norm. However, the uniqueness of solutions and the continuous dependence of solutions with respect to initial data in $H^1(\R^2)$ are unknown. Later, the author in \cite{Huh} obtained the unconditional uniqueness of solutions in $L^{\infty}_t H^1$. Recently, it was shown in \cite{Lim} that the solution map is Lipschitz continuous for  initial data in $H^s(\R^2)$ for any $s \geq 1$. While $p>3$, the local well-posedness of solutions in $H^1(\R^2)$ has not been treated yet. This is the motivation of our study and the result reads as follows.

\begin{thm} \label{thm.wellposedness}
Let $p>3$, then, for any $R >0$, there exists a constant $T_{\textnormal{max}}=T_{\textnormal{max}}(R)>0$ such that \eqref{CSS} has a unique solution
$$
u \in C^1([0,T_{\textnormal{max}}), L^2_{rad}(\mathbb{R}^2))\cap C([0,T_{\textnormal{max}}), H^1_{rad}(\mathbb{R}^2)) \ \  
$$
provided initial datum
$$
u_0 \in H^1_{rad}(\mathbb{R}^2), \ \ \|u_0\|_{H^1(\mathbb{R}^2)} \leq R.
$$
Moreover, we have the following properties.
\begin{itemize}
\item[(i)] The mass and energy are conserved, i.e. for any $t \in [0, T_{\textnormal{max}})$,
$$
M(u(t))=M(u_0), \quad E(u(t))=E(u_0),
$$
where
\begin{align*}
M(u)= \int_{\R^2} |u|^2 \, dx
\end{align*}
and
\begin{align*}
E(u)= \frac 12 \int_{\R^2} |D_1 u|^2 + |D_2 u|^2 \, dx - \frac{1}{p+1} \int_{\R^2} |u|^{p+1} \, dx.
\end{align*}
\item[(ii)] The solution map
$$
u_0 \in H^1_{rad}(\mathbb{R}^2) \mapsto u  \in  C^1([0,T_{\textnormal{max}}), L^2_{rad}(\mathbb{R}^2))\cap C([0,T_{\textnormal{max}}), H^1_{rad}(\mathbb{R}^2))
$$
is Lipschitz continuous in the ball of radius $R$ in $H^1_{rad}(\mathbb{R}^2)$.
\item[(iii)] Blow-up alternative holds: either $T_{\textnormal{max}}= +\infty$ or $\|\nabla u(t)\|_{L^2(\R^2)}=+\infty$ as $t \to T_{\textnormal{max}}^-$.
\end{itemize}
\end{thm}


To prove Theorem \ref{thm.wellposedness}, we first establish estimates for the electromagnetic potentials. Then we employ the contraction mapping principle along with Strichartz-type estimates to derive the desired results.

We next discuss dynamical behaviors of solutions to the Cauchy problem for \eqref{CSS} in the mass supercritical case, which has not been dealt with so far. Concerning this topic in the mass critical case, we refer the readers to \cite{BBS, Hu1, KKO} for finite time blow-up of solutions and \cite{LL, LS16, OhPu} for scattering of solutions. Let us mention that the authors in \cite{LS16} obtained the scattering of solutions in $L^2(\R^2)$ by applying the concentration compactness arguments due to Kenig and Merle \cite{KM1, KM2}.

Before stating our result in this direction, we need to introduce some notations. The underlying energy functional is defined by
$$
S(u) = E(u) + \frac 12 M(u).
$$
The ground state energy is defined by
\begin{align*}
d= \inf \left\{ S(u): u \in H^1_{rad}(\R^2)  \backslash \{0\}, \, K(u)=0\right\},
\end{align*}
where
$$
K(u)= \int_{\R^2} |D_1 u|^2 + |D_2 u|^2 \, dx - \frac{p-1}{p+1} \int_{\R^2} |u|^{p+1} \, dx.
$$
For $p>3$, we are concerned with dynamical behaviors of solutions to the Cauchy problem for \eqref{CSS} with initial data starting from the following two sets,
$$
\mathcal{K}^+= \left\{u  \in H^1_{rad}(\R^2) : S(u) <d, \, K(u)  >0\right\}
$$
and
$$
\, \, \, \mathcal{K}^-= \left\{u  \in H^1_{rad}(\R^2) : S(u) <d, \, K(u) <0\right\}.
$$

\begin{thm} \label{thm.scattering1}
Let $p>3$ and let $u$ be the solution to the Cauchy problem for \eqref{CSS} with initial datum $u_0 \in H^1_{rad}(\R^2)$.
\begin{itemize}
\item[i)] If $u_0 \in \mathcal{K}^+$, then $u$ exists globally in time and scatters in $H^1_{rad}(\R^2)$, i.e. there exist $u_{\pm} \in H_{rad}^1(\R^2)$ such that
$$
\lim_{t \to \pm \infty} \|u(t)- e^{\textnormal{i} \Delta t} u_{\pm}\|_{H^1(\R^2)} =0.
$$
\item[ii)] If $u_0 \in \mathcal{K}^-$ and $p \leq 5$, then $u$ blows up in finite time.
\end{itemize}
\end{thm}

Let us now sketch the strategy to prove Theorem \ref{thm.scattering1}.  The global existence of solutions follows directly from the conservation laws and Gagliardo-Nirenberg type inequality. The proof of scattering of solutions is mainly inspired by the approach developed in \cite{ADM}, which avoids the use of the concentration-compactness-rigidity arguments in the spirit of Kenig and Merle \cite{KM1, KM2}. The key ingredient is to establish Morawetz-type estimates for solutions, from which we are able to deduce the scattering of solutions. In \cite{ADM}, the authors considered the following classical NLS in two space dimensions,
$$
(\textnormal{i} \, \partial_t + \Delta )u= -|u|^{p-1}u.
$$
In view of Morawetz-type estimates, the authors obtained that
$$
\int_{t_0}^{t_1} \int_{\R^2} |u|^{p+1} \, dx dt \lesssim \eps
$$
for a suitable large time interval $[t_0,t_1]$. From the smallness of the term
\begin{align*}
\left\|u\right\|_{L^{p+1}\left([t_1,\infty); L^{p+1}(\mathbb{R}^{2})\right)},
\end{align*}
then the scattering of solutions follows. In our case, the discussion of scattering of solutions becomes complex and the nonlocal terms require specific treatments. First we need to derive that
\begin{align*}
\int_{t_0}^{t_1} \int_{\R^2} |u|^{p+1} \, dx dt +\int_{t_0}^{t_1} \int_{\R^2} \left(\frac{A_{\theta}(|u|^2)}{r}\right)^2 |u|^2  + \int_{t_0}^{t_1} \int_{\R^2} A_0(|u|^2)|u|^2 \, dx dt \lesssim \eps,
\end{align*}
where $r=|x|$, $A_{\theta}$ and $A_0$ are electromagnetic potentials defined by \eqref{eq.ss2a}. Next we have to establish the smallness of the term
\begin{align*}
\left\|u\right\|_{L^{q}\left([t_1,\infty); L^r(\mathbb{R}^{2})\right)}
\end{align*}
for the parameters $q, r>2$ in the whole triangle
$$
\mathfrak{T}:= \left\{\left(\frac{1}{q}, \frac{1}{r}\right) : 0 < \frac{1}{q}+\frac{1}{r} \leq \frac{1}{2}, \ 2 < q,r < \infty \right\}.
$$
Thus we have the scattering of solutions. The proof of blow-up of solutions benefits from variational characterization of the ground state energy and analysis of the evolution of a localized virial quantity.

\subsection*{Structure of the Paper} This paper is organized as follows. In Section \ref{pre}, we present some preliminary results used to prove our main results. In Section \ref{proof1}, we first establish useful estimates of the electromagnetic potentials $A_{\theta}$ and $A_0$, we then prove Theorem \ref{thm.wellposedness}. Section \ref{proof23} is devoted to the proof of Theorem \ref{thm.scattering1}. In Appendix, we demonstrate that the smallness of initial data implies the smallness of solutions and we also establish a crucial estimate for a Nemytskii operator.

\subsection*{Notations} \textit{Throughout the paper, for any $1 \leq q< + \infty$, we denote by $L^q(\R^2)$ the Lebesgue space equipped with the norm
$$
\|u\|_{L^q(\R^2)}=\left(\int_{\R^2} |u|^q \, dx\right)^{\frac 1 q}.
$$
For any $I \subset \R$ and $1 \leq q, r \leq \infty$, we write
$$
\|u\|_{L^q(I, L^r(\R^2))}=\left\|\|u(t)\|_{L^r(\R^2)}\right\|_{L^q(I)}.
$$
We denote by $H^1(\R^2)$ the usual Sobolev space equipped with the standard norm
$$
\|u\|_{H^1(\R^2)}=\|u\|_{L^2(\R^2)}+ \|\nabla u\|_{L^2(\R^2)}.
$$
In addition, $H^1_{rad} (\R^2)$ stands for the subspace of $H^1(\R^2)$, which consists of the radially symmetric functions in $H^1(\R^2)$. For any nonnegative quantities $A, B$, we use the notations $A \lesssim B$  to denote $A \leq C B$ for some $C > 0$.
}

\section{Preliminaries} \label{pre}

In this section, we present some preliminary results used to establish our main results. Let us first show the well-known Gagliardo-Nirenberg inequality in $H^1(\R^2)$,
\begin{align} \label{GN}
\|u\|_{L^q(\R^2)}  \lesssim \|\nabla u\|_{L^2(\R^2)}^{\alpha} \|u\|_{L^2(\R^2)}^{1- \alpha} \quad \mbox{for} \, \, \alpha=1-\frac 2q.
\end{align}
Let us also give the Strauss inequality in $H^1_{rad}(\R^2)$,
\begin{align} \label{Strauss}
\left\||x|^{\frac 12} u \right\|_{L^{\infty}(\R^2)} \lesssim \|u\|_{H^1(\R^2)}.
\end{align}

\begin{lem} \cite[Theorem 7.21]{Lieb} \label{lem.diaineq}
Let $A: \R^n \to \R^n$ be in $L_{loc}^2(\R^n)$ and let $u$ be in $H^1_A$, where
$$
H^1_A=\left\{ u \in L^2(\R^n): \|\left(\nabla + \textnormal{i} A\right) u \|_{L^2(\R^2)} < \infty \right\},
$$
then $|u| \in H^1(\R^n)$ and the diamagnetic inequality
$$
|\nabla |u|(x)| \leq |(\nabla + \textnormal{i} A) u(x)|
$$
holds pointwise for almost every $x \in \R^n$.
\end{lem}

As an immediate consequence of \eqref{GN} and Lemma \ref{lem.diaineq}, we have that
\begin{align} \label{MGN}
\begin{split}
\|u\|_{L^q(\R^2)} & \lesssim \|\nabla |u|\|_{L^2(\R^2)}^{\frac{q-2}{q}} \|u\|_{L^2(\R^2)}^{\frac 2 q} \\
& \lesssim \left( \|D_1 u\|^2_{L^2(\R^2)} + \|D_2 u\|^2_{L^2(\R^2)}\right)^{{\frac{q-2}{2q}}} \|u\|_{L^2(\R^2)}^{\frac{2}{q}}.
\end{split}
\end{align}

\begin{defi}
A couple $(q,r)$ of positive numbers $q \geq 2$ and $r \geq2 $ is called $2d$ Schr\"{o}dinger admissible,
if it satisfies $(q,r) \neq (2,\infty) $ and
\begin{align*}
\frac{1}{q}+\frac{1}{r}=\frac{1}{2}. 
\end{align*}
\end{defi}

\begin{lem} {\label {lem2.1}} (\cite{KeTao98, S01, T00})
\rm{(Homogeneous Strichartz Estimates)}
There hold the following estimates.
\begin{itemize}
\item[i)] If $(p,q)$ is a $2d$ Schr\"{o}dinger admissible couple, then
\begin{align*}
\left\|e^{i t\Delta}f\right\|_{L^{q}(\mathbb{R};L^{r}(\mathbb{R}^{2}))}\lesssim\|f\|_{L^{2}(\mathbb{R}^{2})}.
\end{align*}
 \item[ii)] If $f$ is radial, then
  \begin{align*}
\left\|e^{i t\Delta}f\right\|_{L^{2}(\mathbb{R};L^{\infty}(\mathbb{R}^{2}))}\lesssim\|f\|_{L^{2}(\mathbb{R}^{2})}.
\end{align*}
\end{itemize}
\end{lem}

\begin{lem} (\cite{KeTao98, S01, T00}) \label {lem2.2}
\rm{(Inhomogeneous  Strichartz Estimates)} Let $(r_{j},q_{j})$ be $2d$ Schr\"{o}dinger admissible couples for $j=1, 2$. Then the following estimates hold.
\begin{itemize}
  \item[i)] There holds that
\begin{align*}
\left\|\int_{-\infty}^t e^{\textnormal{i}(t-s)\Delta}F(s,\cdot)ds\right\|_{L^{q_{2}}(\mathbb{R}, L^{r_{2}}(\mathbb{R}^{2}))}\lesssim\|F\|_{L^{q_{1}^{\prime}}(I , L^{r_{1}^{\prime}}(\mathbb{R}^{2}))}, 
\end{align*}
where $q_{j}^{\prime}, r_{j}^{\prime}$ are the dual exponents of $q_{j}$ and $r_{j}$, respectively.
 \item[ii)] If $F$ is radial, then
\begin{align} \label{eq.r1}
\begin{split}
& \left\|\int_{-\infty}^{\infty} e^{\textnormal{i}(t-s)\Delta}F(s,\cdot)ds\right\|_{L^{2}( \mathbb{R} , L^{\infty}(\mathbb{R}^{2}))}
+ \left\|\int_{-\infty }^{\infty} e^{\textnormal{i}(t-s)\Delta}F(s,\cdot)ds\right\|_{L^{\infty}(\mathbb{R}, L^{2}(\mathbb{R}^{2}))} \\
& \lesssim \inf \left\{\|F\|_{L^{2}(\mathbb{R} , L^{1}(\mathbb{R}^{2}))} , \|F\|_{L^{1}(\mathbb{R} , L^{2}(\mathbb{R}^{2}))}\right\}.
\end{split}
\end{align}
\end{itemize}
\end{lem}

\begin{rem}
We can apply Christ-Kiselev's lemma from \cite{CK01} to deduce that
\begin{align*} 
& \left\|\int_{-\infty}^t e^{\textnormal{i}(t-s)\Delta}F(s,\cdot)ds\right\|_{L^{q_{2}}(\mathbb{R}, L^{r_{2}}(\mathbb{R}^{2}))} \lesssim  \|F\|_{L^{2}(\mathbb{R} , L^{1}(\mathbb{R}^{2}))}
\end{align*}
and
\begin{align*} 
& \left\|\int_{-\infty}^{t} e^{\textnormal{i}(t-s)\Delta}F(s,\cdot)ds\right\|_{L^{2}( \mathbb{R} , L^{\infty}(\mathbb{R}^{2}))}
 \lesssim \|F\|_{L^{q_{1}^{\prime}}(I , L^{r_{1}^{\prime}}(\mathbb{R}^{2}))}
\end{align*}
for any $2d$ Schr\"odinger admissible couples $(q_j,r_j)$ and $j=1,2.$
\end{rem}

\begin{rem}
For any two intervals $I, J \subseteq \mathbb{R}$, it follows from \eqref{eq.r1} that
\begin{align*} 
& \left\|\int_{I} e^{\textnormal{i}(t-s)\Delta}F(s,\cdot)ds\right\|_{L^{2}( J , L^{\infty}(\mathbb{R}^{2}))}
+ \left\|\int_{I }e^{\textnormal{i}(t-s)\Delta}F(s,\cdot)ds\right\|_{L^{\infty}(J, L^{2}(\mathbb{R}^{2}))} \\
& \lesssim \inf \left\{\|F\|_{L^{2}(I , L^{1}(\mathbb{R}^{2}))} , \|F\|_{L^{1}(I , L^{2}(\mathbb{R}^{2}))}\right\}.
\end{align*}
\end{rem}

\section{Proof of theorem \ref{thm.wellposedness}} \label{proof1}

Our aim in this section is to establish Theorem \ref{thm.wellposedness}. To do this, we shall make use of the following radial ansatz,
\begin{align}\label{eq.lw1}
\begin{split}
u(t,x)&=u(t,\mid x\mid ), \qquad  \quad \, A_0(t,x) = A_0(t,\mid x\mid ), \\
A_1(t,x) &= -\frac{x_2}{\mid x\mid } h(t,\mid x\mid ), \quad A_2(t,x) = \frac{x_1}{\mid x\mid } h(t,\mid x\mid ).
\end{split}
\end{align}
Setting $r=\mid x\mid $ and
$$
A_{\theta} = -x_2A_1+x_1A_2 = -\mid x\mid h(t,\mid x\mid ),
$$
then we obtain that
\begin{align*}
A_1(t,x) &= \frac{x_2}{\mid x\mid ^2} A_\theta(t,\mid x\mid ), \quad A_2(t,x) = -\frac{x_1}{\mid x\mid ^2} A_\theta(t,\mid x\mid )
\end{align*}
and
\begin{align}
\left\{
\begin{aligned}
 A_1 \partial_1 u+ A_2 \partial_2 u&= 0, \\
 (\partial_1 + \textnormal{i} A_1)^2 u +  (\partial_2 + \textnormal{i} A_2)^2 u&= \Delta u  - \frac{A_{\theta}^2}{r^2} u, \\ \label{eq.gg3}
\mid  (\nabla +\textnormal{i} A) u \mid ^2 &=  \mid  \nabla  u \mid ^2  + \frac{A_{\theta}^2}{r^2} \mid u\mid ^2,
\end{aligned}
\right.
\end{align}
where $A=(A_1, A_2)$. In addition, it follows from \eqref{CSS} that
\begin{align} 
\left\{
\begin{aligned} \label{sys1}
\textnormal{i} \, \partial_t u+ \Delta u &=  A_0 u + \frac{A_{\theta}^2}{r^2} u -\mid u\mid ^{p-1} u,\\
\partial_r A_0 & = \frac{A_{\theta}}{r} \mid u\mid ^2, \\
\partial_t A_{\theta}& = r \mbox{Im}\left(\overline{u} \,\partial_r u\right), \\
\partial_r A_{\theta} &=- \frac{r}{2} \mid u\mid ^2.
\end{aligned}
\right.
\end{align}
In view of \eqref{sys1}, the terms $A_0, A_{\theta}$ can be related explicitly to $u$ and we have that
\begin{align}\label{eq.ss2a}
A_{\theta} (t, r) = A_{\theta}(\mid u(t,\cdot)\mid ^2)(r), \quad  A_0(t, r) = A_0(\mid u(t,\cdot)\mid ^2)(r),
\end{align}
where
\begin{align*}
A_\theta(f)(r)=- \frac{1}{2}\int_0^r f(\rho) \rho \,d\rho, \quad
A_0(f)(r)=-\int_r^\infty \frac{A_\theta(f)}{\rho} f(\rho)  \, d\rho.
\end{align*}
Thus \eqref{sys1} leads to
\begin{align}\label{NLS1}
\textnormal{i} \, \partial_t u+ \Delta u = \Lambda(u),
\end{align}
where
$$
\Lambda(u) = \Lambda_1(u) + \Lambda_2(u) + \Lambda_3(u),$$
and
\begin{align} \label{defLa}
\Lambda_1(u) = A_0(\mid u\mid ^2) u, \ \ \Lambda_2(u) = \left(\frac{A_\theta(\mid u\mid ^2)}{r}\right)^2 u,\ \ \Lambda_3(u)=- \mid u\mid ^{p-1} u.
\end{align}
As a consequence of \eqref{NLS1}, we see that
\begin{align}\label{NLS1i}
u(t) = e^{\mathrm{i} \Delta t} u_0 + \mathrm{i} \int_0^t e^{\mathrm{i} \Delta (t-s)} \Lambda(u)(s) \, ds.
\end{align}

\subsection{Estimates of the electromagnetic potentials.} In the following, we shall present some apriori estimates for the nonlocal terms $ A_0 $ and $ A_{\theta}$, which may be of independent interests.

\begin{lem}  \label{lem.ineq}
If  $2 \leq  s   \leq \infty$ and $b \in [0,2]$ satisfy
$$
\frac b2 = 1- \frac 2s + \frac 1q,\quad s > 2
$$
or $b=0$, $q=\infty$ and $s=2$, then
\begin{align}\label{eq.es0}
\left\| \frac{A_\theta(\mid u\mid ^2)}{\mid x\mid ^b}\right\| _{L^q(\R^2)} \lesssim \|d u\| ^2_{L^{s} (\mathbb{R}^2)}.
\end{align}
\end{lem}

\begin{proof}
In order to achieve \eqref{eq.es0}, we shall set $ f = \mid u\mid ^2$ and prove the estimate
\begin{align}\label{eq.es310}
\left\| \frac{A_\theta(f)}{\mid x\mid ^b}\right\| _{L^q(\R^2)} \lesssim \| f\| _{L^{\frac s 2} (\R^2)}.
\end{align}
To do this, we shall employ interpolation arguments. Let us first take three points $B_1 (0,0),$ $B_2 (0,1/2)$ and $B_3(1,1/2)$ in the plane $(1/q,1/s)$.
In the point $B_1$, we have that $b=2$. Using the simple $L^\infty$-estimate involving the maximal function, we arrive at
$$ \left\| \frac{A_\theta(f)}{\mid x\mid ^2}\right\| _{L^\infty(\R^2)} = \left\| M(f)\right\| _{L^{\infty}(\R^2)} \lesssim \left\| f\right\| _{L^{\infty}(\R^2)}.
$$
In the point $B_2$, we see that $b=0$. Therefore, we obtain the trivial estimate
\begin{align*}
\left\| A_\theta(f)\right\| _{L^\infty(\R^2)} \lesssim \left\| f\right\| _{L^1(\R^2)}.
\end{align*}
Finally, in the point $B_3$, we get that $b=2$. Note that
$$
\bigl\lvert  \frac{A_\theta(f)}{\mid x\mid ^2} \bigr\rvert  \lesssim \mid M(f)(x)\mid , \quad \| M(f) \| _{L^{1,\infty}(\R^2)} \lesssim  \| f\| _{L^{1}(\R^2)}.
$$
This estimate implies that for any $r \in (1,\infty)$,
\begin{align}\label{eq.ipet1}
\left\|  \frac{A_\theta(f)}{\mid x\mid ^2} \right\| _{L^r(\mathbb{R}^2)}  \lesssim \left\|  M(f)\right\| _{L^r(\mathbb{R}^2)} \lesssim \left\| f\right\| _{L^r(\mathbb{R}^2)}.
\end{align}
Consequently, we find that
\begin{align*}
\left\| \frac{A_\theta(f)}{\mid x\mid ^2} \right\| _{L^{1,\infty}(\R^2)} \lesssim \left\| f\right\| _{L^1(\R^2)}.
\end{align*}
An interpolation argument based on the Marcinkiewicz interpolation theorem leads to the inequality of Lemma \ref{lem.ineq} for all points $(1/s,1/q)$ in the closure of $\triangle B_1 B_2 B_3$ except the segment $B_2B_3$ (including $B_3$ and excluding $B_2$). This completes the proof.
\end{proof}

For radial functions $f,g \in L^1(\mathbb{R}^2) \cap L^\infty(\mathbb{R}^2)$, we define the following bilinear form,
\begin{align*}
A_0(f,g)(r) = - \int_r^\infty \frac{A_\theta(f)}{\rho} g(\rho)  \, d\rho.
\end{align*}
It is clear to see that $A_0(f, f)=A_0(f)$.
\begin{lem}  \label{lem.bf1}
If $q \in [1,\infty], s_1 \in [2,\infty]$ and $s_2 \in [2,\infty]$ satisfy
\begin{align}\label{eq.bf2}
\renewcommand*{\arraystretch}{1.9}
\left\{ \begin{array}{l}
q < \infty , s_1 >2, \  \frac{1}{q} \leq \frac 2{s_1} + \frac 2{s_2} \leq 1 + \frac 1q, \\
\frac{a}{2} + b = 1 - \frac{2}{s_1}-\frac{2}{s_2}+ \frac{1}{q}, \ a <  \frac{2}{q},
\end{array} \right.
\end{align}
or
\begin{align}\label{eq.bf2mu1}
\renewcommand*{\arraystretch}{1.9}
\left\{ \begin{array}{l}
q=\infty ,\ s_1>2, \  0 \leq \frac 2{s_1} + \frac 2{s_2} \leq 1,\\
\frac{a}{2} + b = 1 - \frac{2}{s_1}-\frac{2}{s_2}, \ a \leq 0,
\end{array} \right.
\end{align}
or
\begin{align}\label{eq.bf2mm1}
 q=\infty,s_1=s_2=2, \  b=0, a=-2,
\end{align}
then
\begin{align}\label{eq.bf3}
\left\| \frac{A_0(\mid u\mid ^2)}{\mid x\mid ^a}\right\| _{L^q(\R^2)} \lesssim \| u\| ^2_{L^{s_1} (\R^2)}  \left\| \mid x\mid ^{b} u\right\| ^2_{L^{s_2} (\R^2)}.
\end{align}
\end{lem}
\begin{rem}\label{r.noH1}
For the case $q=\infty$ in \eqref{eq.bf2mu1}, we can avoid the assumption $a < 2/q=0$ in \eqref{eq.bf2}. In fact, this condition comes from the application of Hardy type inequality and we use it only for $q < \infty.$
\end{rem}

\begin{proof}
To prove \eqref{eq.bf3}, we shall deduce that one of the assumptions  \eqref{eq.bf2},  \eqref{eq.bf2mu1} and  \eqref{eq.bf2mm1} implies
\begin{align}\label{eq.bf3m}
\left\| \frac{A_0(f,g)}{\mid x\mid ^a}\right\| _{L^q(\R^2)} \lesssim \| f\| _{L^{\frac{s_1}{2}} (\R^2)}  \left\| \mid x\mid ^{2b} g\right\| _{L^{\frac{s_2}{2}} (\R^2)}.
\end{align}
Let us first consider the case $ q < \infty$ and \eqref{eq.bf2} is satisfied. In this case, we are going to take advantage of the following Hardy type inequality from \cite[Theorem 2]{M72},
\begin{align}\label{eq.M1.41}
\left(\int_{0}^{\infty}\mid r^{A} \int_{r}^{\infty} F(\rho) \,d\rho\mid ^{q} \,d r \right)^{\frac 1  q} \lesssim \left(\int_{0}^{\infty}\mid \rho^{A+1} F(\rho)\mid ^{q} \,d \rho\right)^{\frac 1 q},
\end{align}
where $1 \leq q < \infty$, $-1/q<A<\infty$ and $F: \mathbb{R}\to \mathbb{R}$ is a Lebesgue measurable function. As an immediate consequence of \eqref{eq.M1.41}, we have that
\begin{align}\label{eq.M1.5}
\left\| r^{\alpha} \int_r^\infty F(\rho) \, d\rho \right\| _{L^q(rdr)} \lesssim \left\| r^{\alpha+1} F(r) \right\| _{L^q(rdr)},
\end{align}
where $1 \leq q <\infty$,  $\alpha > -{2}/{q}$ and
$$
\left\| G\right\| _{L^q(r dr)} := \left(\int_0^\infty \mid G(r)\mid ^q  r\, dr \right)^{1/q}.
$$
In addition, for $\alpha \geq 0$, by making an interpolation between \eqref{eq.M1.5} and the trivial estimate
\begin{align*}
\left\| r^{\alpha} \int_r^\infty F(\rho) \, d\rho \right\| _{L^\infty(rdr)} \lesssim \left\| r^{\alpha-1} F(r) \right\| _{L^1(rdr)},\end{align*}
we obtain that
\begin{align}\label{eq.M1.6}
\left\| r^{\alpha} \int_r^\infty F(\rho) \,d\rho \right\| _{L^q(rdr)} \lesssim \left\| r^{\alpha+1 + \frac{2}{q}-\frac{2}{m}} F(r) \right\| _{L^m(rdr)},
\end{align}
where $1 \leq m \leq q < \infty$. With these inequalities in hand, we now come back to the proof of \eqref{eq.bf3}. Let us first treat the case $0<a<2/q$. By  using \eqref{eq.M1.5} with $\alpha = -a$ and  H\"older's inequality, we find that
\begin{align*}
\left\| \frac{A_0(f,g)}{\mid x\mid ^a}\right\| _{L^q(\mathbb{R}^2)}&= \left\| r^{-a} A_0(f,g)(r)\right\| _{L^q(rdr)}
= \left\| r^{-a} \int_r^\infty \frac{A_\theta(f)}{\rho} g(\rho) \,  d \rho \right\| _{L^q(rdr)} \\
&\lesssim \left\| r^{-a} A_\theta(f)(r) g(r)  \right\| _{L^q(rdr)} \leq \left\| \mid x\mid ^{-a_1} A_\theta(f)   \right\| _{L^m(\mathbb{R}^2)}
\left\|  \mid x\mid ^{2b}g  \right\| _{L^{\frac{s_2}{2}}(\mathbb{R}^2)},
\end{align*}
where $a_1=a+2b$ and
$$
\frac 1 m +\frac {2}{s_2} = \frac 1 q.
$$
In addition, from the estimate \eqref{eq.es310}, we get that
$$
\left\| \frac{A_\theta(f)}{\mid x\mid ^{a_1}}\right\| _{L^m(\mathbb{R}^2)} \lesssim \| f\| _{L^{\frac{s_1}{2}}(\mathbb{R}^2)},
$$
where
$$
\frac{a_1}{2}= 1 - \frac{2}{s_1} + \frac{1}{m}, \ \frac 1 m +\frac {2}{s_2} =\frac 1q, \ a_1 \in [0,2], \ s_1 \in (2,\infty].
$$
As a consequence, we arrive at \eqref{eq.bf3m} provided
$$
\frac{a}{2} + b =  1 - \frac{2}{s_1} -  \frac{2}{s_2} + \frac{1}{q}, \ a + 2b \in [0,2], \ s_1 \in (2,\infty].
$$
We next handle the case $\alpha \leq 0$. In this case, we can apply the estimate \eqref{eq.M1.6} with $\alpha =-a \geq 0$ to conclude that
\begin{align*}
\left\| \frac{A_0(f,g)}{\mid x\mid ^a}\right\| _{L^q(\mathbb{R}^2)}= \left\| r^{-a} A_0(f,g)(r)\right\| _{L^q(rdr)}
&= \left\| r^{-a} \int_r^\infty \frac{A_\theta(f)}{\rho} g(\rho) \,  d \rho \right\| _{L^q(rdr)} \\
& \lesssim \left\| r^{-a+\frac{2}{q} -\frac{2}{q_1}} A_\theta(f)(r) g(r)  \right\| _{L^{q_1}(rdr)},
\end{align*}
where $ 1 \leq q_1 \leq q < \infty$. Taking $q_1=1$  and using H\"older's inequality, then we get that
$$ \left\| \frac{A_0(f,g)}{\mid x\mid ^a}\right\| _{L^q(\mathbb{R}^2)} \leq \left\| \mid x\mid ^{-a-2+\frac{2}{q}-2b } A_\theta(f)\right\| _{L^m(\mathbb{R}^2)}
\left\|  \mid x\mid ^{2b}g  \right\| _{L^{\frac {s_2}{2}}(\mathbb{R}^2)},
$$
where
$$
\frac {1} {m}+\frac {2} {s_2}=1.
$$
Applying again the estimate \eqref{eq.es310}, we obtain that
$$
\left\| \frac{A_\theta(f)}{\mid x\mid ^{a_1}}\right\| _{L^m(\mathbb{R}^2)} \lesssim \| f\| _{L^{\frac {s_1}{2}}(\mathbb{R}^2)},
$$
where $a_1=a+2-\frac{2}{q}+2b$ and
$$
\frac{a+2-\frac{2}{q}+2b}{2}= 1 - \frac{2}{s_1} + \frac{1}{m}, \  \frac{1}{m}+\frac{2}{s_2}=1, \ a_1 \in [0,2], \ s_1 \in (2,\infty]. $$
It is equivalent to
$$
\frac{a}{2} + b =  1 - \frac{2}{s_1} -  \frac{2}{s_2}+\frac{1}{q}, \ a + 2b \in [0,2], \ s_1 \in (2,\infty].
$$
Therefore, we have that \eqref{eq.bf3m} holds when \eqref{eq.bf2} is fulfilled. Next we consider the case $q=\infty$ when \eqref{eq.bf2mu1} or \eqref{eq.bf2mm1} is satisfied. By applying H\"older's inequality, we first see that
\begin{align*}
\renewcommand\arraystretch{2.2}
r^{-a}\mid  A_0(f,g)(r)\mid  \lesssim \left \| \mid x\mid ^{-2-a} A_\theta(f) g \right \| _{L^1(\mathbb{R}^2)} \lesssim \left\| \mid x\mid ^{-2-a-2b}A_\theta(f) \right\| _{L^{m}(\mathbb{R}^2)} \left\|  \mid x\mid ^{2b}  g\right \| _{L^{\frac{s_2}{2}}(\mathbb{R}^2)},
\end{align*}
where
$$
\frac 1 m + \frac {2}{s_2}=1.
$$
Moreover, according to the estimate \eqref{eq.es310}, we get that
$$
\left\| \frac{A_\theta(f)}{\mid x\mid ^{a_1}} \right\| _{L^{m}(\mathbb{R}^2)} \lesssim \|  f \| _{L^{\frac{s_1}{2}}(\mathbb{R}^2)},
$$
where $a_1=2+a+2b$ and
\begin{align*}
\frac{2+a+2b}2 = 1- \frac{2}{s_1} + \frac 1 m, \ \frac{1}{m}+\frac{2}{s_2}=1, \  2+a+2b\in [0, 2] ,  \ s_1  \in (2, \infty],
\end{align*}
or
\begin{align*}
m = \infty, \ s_1=2, \ \frac{2+a+2b}2 = 0.
\end{align*}
Hence we have that \eqref{eq.bf3m} holds when one of the assumptions \eqref{eq.bf2mu1} or \eqref{eq.bf2mm1} is fulfilled. Thus the proof is completed.
\end{proof}

\begin{cor}\label{cor.a01} If $u \in H^1_{rad}(\mathbb{R}^2),$ then for
 $q \in [1,\infty) $ and $a \in (-2,2/q)$ or for $q=\infty $ and $a \in [-2,0]$, there holds that
\begin{align}\label{eq.cub2}
\left\| \frac{A_0(\mid u\mid ^2)}{\mid x\mid ^a}\right\| _{L^q(\mathbb{R}^2)} \lesssim \| u\| ^4_{H^1 (\mathbb{R}^2)}  .
\end{align}
\end{cor}

\begin{proof}
We start with considering the case $ q\in[1, \infty)$ and $a \in (-2,2/q)$. According to Lemma \ref{lem.bf1}, it suffices to deduce that, for $q \in [1,\infty) $ and $a \in (-2,2/q)$, there exist $s_1, s_2 \in [2,\infty]$ and $b \in \R$ such that the assumption \eqref{eq.bf2} is fulfilled and
\begin{align}\label{eq.cra1}
\left\| \mid x\mid ^{b} u\right\| _{L^{s_2}(\mathbb{R}^2)}\lesssim \| u\| _{H^1(\mathbb{R}^2)}.
\end{align}
Let us first prove that $\eqref{eq.cra1}$ holds for some $b \in \R$. Using standard Strauss type estimate, we know that
\begin{align}\label{trig}
\left\| x\mid ^{\frac 1 2} u\right\| _{L^\infty(\mathbb{R}^2)}\lesssim\| u\| _{H^s(\mathbb{R}^2)}, \ s\in (1/2,1].
\end{align}
By making an interpolation between \eqref{trig}  and $\| u\| _{L^2} = \| u\| _{L^2}$, then we derive that
$$
\left\| \mid x\mid ^{\frac {\theta}{2}} u\right\| _{L^{\frac{2}{1-\theta}}(\mathbb{R}^2)}\lesssim \| u\| _{H^{\theta}(\mathbb{R}^2)}, \ \theta \in (0,1).
$$
Therefore, we get that
$$
\left\| \mid x\mid ^{\frac 1 2-\frac {1}{s_2}} u\right\| _{L^{s_2}(\mathbb{R}^2)} \lesssim \| u\| _{H^{1}(\mathbb{R}^2)}, \ s_2  \in [2,\infty).
$$
From such an estimate and the Sobolev embedding inequality
$$
\|  u\| _{L^{s_2}(\mathbb{R}^2)} \lesssim \| u\| _{H^{1}(\mathbb{R}^2)}, \ s_2  \in [2,\infty),
$$
then we infer that \eqref{eq.cra1} holds for $b \in \R$ satisfying
\begin{align} \label{b}
0 \leq b \leq \frac 1 2-\frac {1}{s_2}.
\end{align}
Next we are going to deduce that there exist $s_1, s_2 \in [2,\infty]$ and $b \in \R$ satisfying \eqref{b} such that the assumption \eqref{eq.bf2} is fulfilled. For $a \in [0, 2/q)$, by taking $0 < \varepsilon < 4/(q+1)$ and $s_1=s_2=4-\varepsilon$, we can see that the condition  \eqref{eq.bf2} with $b=0$ is fulfilled, i.e.
\begin{align*}
\renewcommand*{\arraystretch}{1.9}
\left\{ \begin{array}{l}
q < \infty , s_1 >2, \  \frac{1}{q} \leq \frac 2{s_1} + \frac 2{s_2} \leq 1 + \frac 1q, \\
\frac{a}{2}  = 1 - \frac{2}{s_1}-\frac{2}{s_2}+ \frac{1}{q}, \ a <  \frac{2}{q}.
\end{array} \right.
\end{align*}
For $a <0$, it is easy to derive that there exist $s_1, s_2 \in (2,\infty)$ such that the condition \eqref{eq.bf2} with $b={1}{/2}-{1}/{s_2}$ is fulfilled, i.e.
\begin{align*}
\renewcommand*{\arraystretch}{1.9}
\left\{ \begin{array}{l}
\frac{1}{q} \leq \frac 2{s_1} + \frac 2{s_2} \leq 1 + \frac 1q, \\
\frac{2}{s_1}+\frac{1}{s_2}  = \frac{1}{2} + \frac{1}{q} - \frac{a}{2}.
\end{array} \right.
\end{align*}
We now consider the case $q=\infty$ and $a \in [-2, 0]$. In this case, we know that
$$
r^{-a}\mid  A_0(r)\mid  \lesssim \| \mid x\mid ^{-2-a} A_\theta(\mid u\mid ^2) \mid u\mid ^2  \| _{L^1(\mathbb{R}^2)}.
$$
For $a=0$,  we have that
$$
\left\| \frac{A_\theta(\mid u\mid ^2) \mid u\mid ^2}{\mid x\mid ^2}\right\| _{L^1(\mathbb{R}^2)}\lesssim \| M(\mid u\mid ^2) \mid u\mid ^2  \| _{L^1(\mathbb{R}^2)} \lesssim \| \mid u\mid ^2\| ^2_{L^2(\mathbb{R}^2)} = \| u\| ^4_{L^4(\mathbb{R}^2)}.
$$
For $a=-2$, by using Lemma \ref{lem.ineq}, we have that
$$
\|  A_\theta(\mid u\mid ^2) \mid u\mid ^2  \| _{L^1(\mathbb{R}^2)}\lesssim \|  A_\theta(\mid u\mid ^2)  \| _{L^\infty(\mathbb{R}^2)} \|   \mid u\mid ^2  \| _{L^1(\mathbb{R}^2)} \lesssim \| u\| ^4_{L^2(\mathbb{R}^2)}.
$$
Hence, by applying interpolation argument, we get that \eqref{eq.cub2} holds for $q = \infty $ and $a \in [-2,0].$ This completes the proof.
\end{proof}

We are now in a position to prove Theorem \ref{thm.wellposedness}.

\begin{proof} [Proof of Theorem \ref{thm.wellposedness}]
We shall apply  the contraction mapping principle in Banach space associated with  the following Strichartz norms
\begin{alignat}{2}\label{eq.ss1}
\vertiii{u}_{Str, T,1} = & \vertiii{u}_{Str, T} + \vertiii{\partial_r u}_{Str, T},
\end{alignat}
where
\begin{align*}
\vertiii{u}_{Str, T} = \left\| u \right\| _{L^{2}( [0,T);L^{\infty}(\mathbb{R}^{2}))} + \left\| u \right\| _{L^{\infty}( [0,T);L^{2}(\mathbb{R}^{2}))}.
\end{align*}
For simplicity, we shall write $A_0:=A_0(\mid u\mid ^2)$ and $A_{\theta}:=A_{\theta}(\mid u\mid ^2)$. Let us first show the local existence and uniqueness of solutions to the Cauchy problem for \eqref{CSS} in $H_{rad}^1(\mathbb{R}^2)$. To do this, we need to establish the estimate
\begin{align}\label{eq.ss3}
\begin{split}
& \| \Lambda_j(u)\| _{L^{1}([0,T), L^{2}(\mathbb{R}^2))} + \| \partial_r(\Lambda_j(u)) \| _{L^{1}([0,T), L^{2}(\mathbb{R}^2))}\\
&\lesssim T^{3/4} \left(\vertiii{u}_{Str, T,1}^5+ \vertiii{u}_{Str, T,1}^p\right),
\end{split}
\end{align}
where $T>0$ is a small constant determined later and $\Lambda_j(u)$ are given by \eqref{defLa} for $j=1, 2, 3$.
We begin with proving that \eqref{eq.ss3} holds for $j=1$. From the definition of $\Lambda_1$, it is not hard to find that
\begin{align*}
\| \Lambda_1(u)\| _{L^2(\mathbb{R}^2)} + \| \partial_r(\Lambda_1(u))\| _{L^2(\mathbb{R}^2)}
\leq  \| A_0 u\| _{L^2(\mathbb{R}^2)} + \| A_0 \partial_r u\| _{L^2(\mathbb{R}^2)} +\| u \partial_r A_0\| _{L^2(\mathbb{R}^2)}.
\end{align*}
Note first that the estimate \eqref{eq.bf3} with $a=b=0$, $q=\infty$ and $s_1=s_2=4$ gives that
$ \| A_0\| _{L^\infty(\mathbb{R}^2)} \lesssim \| u\| _{L^4}^4$. So we can write
$$
\| A_0 u\| _{L^2(\mathbb{R}^2)} + \| A_0 \partial_r u\| _{L^2(\mathbb{R}^2)} \lesssim \| u\| ^5_{H^1(\mathbb{R}^2)}.
$$
Observe that
$$
\partial_r A_0=\frac{A_\theta}{r} \mid u\mid ^2
$$
and
$$
\left\| \frac{A_\theta} {r} \right\| _{L^4(\mathbb{R}^2)} \lesssim \| u\| ^2_{L^{8/3}(\mathbb{R}^2)}\lesssim \| u\| ^2_{H^1(\mathbb{R}^2)},
$$
where we used \eqref{eq.es0} with $q=4, b=1$ and $s=8/3$.
As a result, we get that
$$
\| u\partial_r A_0\| _{L^2(\mathbb{R}^2)} = \left\| u \frac{A_\theta}{r} \mid u\mid ^2\right\| _{L^2(\mathbb{R}^2)} \lesssim
\left\|  \frac{A_\theta}{r}\right\| _{L^4(\mathbb{R}^2)} \left\| u \mid u\mid ^2\right\| _{L^4(\mathbb{R}^2)} \lesssim \| u\| ^5_{H^1(\mathbb{R}^2)}.
$$
Consequently, we have that
\begin{align*} 
\| \Lambda_1(u)\| _{L^2(\mathbb{R}^2)} + \| \partial_r\Lambda_1(u)\| _{L^2(\mathbb{R}^2)}  \lesssim \| u\| ^5_{H^1(\mathbb{R}^2)}.
\end{align*}
Integrating this estimate with respect to $t$ on the interval $[0, T)$,  we deduce that
\begin{align*} 
\| \Lambda_1(u)\| _{L^{1}([0,T), L^{2}(\mathbb{R}^2))} + \| \partial_r(\Lambda_1(u)) \| _{L^{1}([0,T), L^{2}(\mathbb{R}^2))} \lesssim T \vertiii{u}_{Str, T,1}^5.
\end{align*}
Further we prove \eqref{eq.ss3} for $j=2$. In light of the definition of $\Lambda_2$, we see that
\begin{align*}
\| \Lambda_2(u)\| _{L^2(\mathbb{R}^2)} + \| \partial_r(\Lambda_2(u))\| _{L^2(\mathbb{R}^2)}
&\lesssim  \left\| \left(\frac{A_\theta}{r}\right)^2 u \right\| _{L^2(\mathbb{R}^2)} +  \left\| \left(\frac{A_\theta}{r}\right)^2 \partial_r u \right\| _{L^2(\mathbb{R}^2)} \\
& \quad +\left\| \frac{A_\theta\partial_r A_\theta}{r^2} u \right\| _{L^2(\mathbb{R}^2)} + \left\| \frac{A_\theta^2}{r^3} u \right\| _{L^2(\mathbb{R}^2)}.
\end{align*}
The first two terms in the right side can be estimated easily by taking $b=1, q=\infty$ and $s=4$ in the estimate \eqref{eq.es0}, i.e.
\begin{align}\label{eq.linb11}
\left\| \frac{A_\theta}{r}\right\| _{L^\infty} \lesssim \| u\| _{L^4}^2\lesssim \| u\| ^2_{H^1(\mathbb{R}^2)}.
\end{align}
From \eqref{eq.linb11}, then we get that
$$
\left\| \left(\frac{A_\theta}{r}\right)^2 u \right\| _{L^2(\mathbb{R}^2)} +  \left\| \left(\frac{A_\theta}{r}\right)^2 \partial_r u \right\| _{L^2(\mathbb{R}^2)} \lesssim \| u\| ^5_{H^1(\mathbb{R}^2)}.
$$
Since
$$
\partial_r A_\theta = -\frac{r}{2}  \mid u\mid ^2,
$$
then we can use \eqref{eq.linb11} to obtain that
$$
\left\| \frac{A_\theta\partial_r A_\theta}{r^2} u \right\| _{L^2(\mathbb{R}^2)} \lesssim \| u\| ^5_{H^1(\mathbb{R}^2)}.
$$
In a similar way, we find from \eqref{eq.es0} that
$$
\left\| \frac{A_\theta}{r^{\frac3 2}}\right\| _{L^\infty} \lesssim \| u\| _{L^8}^2\lesssim \| u\| ^2_{H^1(\mathbb{R}^2)}.
$$
This implies that
$$
\left\| \frac{A_\theta^2}{r^3} u \right\| _{L^2(\mathbb{R}^2)} \lesssim \| u\| ^5_{H^1(\mathbb{R}^2)}.
$$
Therefore, we arrive at
$$ \| \Lambda_2(u)\| _{L^2(\mathbb{R}^2)} + \| \partial_r(\Lambda_2(u))\| _{L^2(\mathbb{R}^2)} \lesssim \| u\| ^5_{H^1(\mathbb{R}^2)}.$$
Integerating this estimate with respect to $t$ on the interval $[0, T)$, then we have that
\begin{align*} 
\| \Lambda_2(u)\| _{L^{1}([0,T), L^{2}(\mathbb{R}^2))} + \| \partial_r(\Lambda_2(u)) \| _{L^{1}([0,T), L^{2}(\mathbb{R}^2))} \lesssim T \vertiii{u}_{Str, T,1}^5.
\end{align*}
The final step is to check that \eqref{eq.ss3} holds for $j=3$. By H\"older's inequality and Sobolev embedding inequality, we can write
$$
\| \mid u\mid ^{p-1}\partial_r u \| _{L^{2}(\mathbb{R}^2)} \leq \| \mid u\mid ^{p-1} \| _{L^{4}(\mathbb{R}^2)} \| \partial_r u \| _{L^{4}(\mathbb{R}^2)} \lesssim
\| u \| _{H^1(\mathbb{R}^2)}^{p-1}  \| \partial_r u \| _{L^{4}(\mathbb{R}^2)}
$$
and
$$
 \int_0^T \| \partial_r u\| _{L^{4}(\mathbb{R}^2)} \,dt \leq T^{3/4} \| \partial_r u \| _{L^{4}([0,T), L^{4}(\mathbb{R}^2))} \leq T^{3/4}\vertiii{u}_{Str, T,1}.
$$
Then we derive that
\begin{align*}
\| \mid u\mid ^p\| _{L^{1}([0,T), L^{2}(\mathbb{R}^2))} + \| \mid u\mid ^{p-1}\partial_r u \| _{L^{1}([0,T), L^{2}(\mathbb{R}^2))} \lesssim T^{3/4}\vertiii{u}_{Str, T,1}^p.
\end{align*}
In conclusion, we have that \eqref{eq.ss3} holds for $j=1,2,3.$

Relying on the estimate \eqref{eq.ss3}, we are now able to take advantage of the contraction mapping principle to deduce the local existence and uniqueness of solutions to the Cauchy problem for \eqref{CSS}. To this end, let us first introduce a Banach space
$$
\mathfrak{B}_T  = \{ u(t,\mid x\mid ) \in C^1([0,T], L^2_{rad}(\mathbb{R}^2))\cap C([0,T], H^1_{rad}(\mathbb{R}^2)) : \vertiii{u}_{Str, T,1} < \infty \}
$$
and define an operator $\mathfrak{L}$ on $\mathfrak{B}_T \to \mathfrak{B}_T$ by
$$
\mathfrak{L} (u) = e^{\mathrm{i} \Delta t} u_0 + \mathrm{i} \int_0^t e^{\mathrm{i} \Delta (t-s)} \Lambda(u)(s) ds.
$$
By applying \eqref{eq.ss3}, we know that the nonlinear term $\Lambda(u)$ satisfies the estimate
\begin{align}\label{eq.ss14}
\begin{split}
&\| \Lambda(u)\| _{L^{1}([0,T], L^{2}(\mathbb{R}^2))} + \| \partial_r(\Lambda(u)) \| _{L^{1}([0,T], L^{2}(\mathbb{R}^2))} \\
&\lesssim T^{3/4} \left(\vertiii{u}_{Str, T,1}^5+\vertiii{u}_{Str, T,1}^p \right).
\end{split}
\end{align}
Taking $u_0 \in H^1_{rad}(\mathbb{R}^2)$ with $\| u_0\| _{H^1(\mathbb{R}^2)} \leq R$, then we deduce from \eqref{eq.ss14}, Lemmas \ref{lem2.1} and \ref{lem2.2} that there exists a constant $C^*>0$ so that
\begin{align}\label{eq.ccp1}
\vertiii{\mathfrak{L} (u)}_{Str,T,1} \leq C^* R + C^* T^{3/4}\left( \vertiii{u}_{Str,T,1}^5+\vertiii{u}_{Str,T,1}^p \right).
\end{align}
Thus the estimate \eqref{eq.ccp1} suggests that the operator $\mathfrak{L}$ maps the ball of radius $2RC^*$ in $\mathfrak{B}_T$ into itself provided $T=T(R)$ is so small that
$$
T^{3/4} \left( (2C^*)^5 R^4 + (2C^*)^p R^{p-1} \right)  < 1.
$$
We next show that $\mathfrak{L}$ is a contraction mapping in $\mathfrak{B}_T$. To do this, it is sufficient to prove the estimate
\begin{align} \label{eq.ss15}
\begin{split}
 &\| \Lambda(u) - \Lambda(v) \| _{L^{1}([0,T), L^{2}(\mathbb{R}^2))} + \| \partial_r(\Lambda(u)-\Lambda(v)) \| _{L^{1}([0,T), L^{2}(\mathbb{R}^2))} \\
 &\lesssim  T^{3/4} (R^4+R^{p-1}) \vertiii{u-v}_{Str, T,1}
 \end{split}
\end{align}
for any $u,v$ in the ball of radius $2C^*R$ in $\mathfrak{B}_T$.
Indeed, this estimate can be verified by a similar way as the proof of \eqref{eq.ss14}, so we omit the details.
Choosing $T=T(R)>0$ sufficiently small, then we can see from \eqref{eq.ss15} that $\mathfrak{L}$ is a contraction mapping in the ball of radius $2RC^*$ in $\mathfrak{B}_T,$ which means that $\mathfrak{L}$ has a unique fixed point $u$ in $\mathfrak{B}_T.$ This in turn yields the local existence and uniqueness of solutions to the Cauchy problem for \eqref{CSS}. We now define that $T_{\textnormal{max}}$ is the maximum existence time of the solution $u$ with the initial datum $u_0$, then the blow-up alternative necessarily follows. By Lemmas \ref{lem2.1} and \ref{lem2.2}, it is not difficult to derive the uniform continuity of the solution mapping. In addition, the conservation laws can be obtained easily by standard ways, see for example \cite{Ca}. Thus we have completed the proof.
\end{proof}

\section{Proof of theorems \ref{thm.scattering1}} \label{proof23}

In this section, we are going to prove Theorem \ref{thm.scattering1}. Above all, for $\lambda>0$, let us introduce a scaling of $u \in H^1(\mathbb{R}^2)$ as
$$
u_{\lambda}(x)=\lambda u(\lambda x) \quad \text{for} \,\,  x \in \mathbb{R}^2.
$$
It is easy to check that $\| u_{\lambda}\| _{L^2(\mathbb{R}^N)}=\| u\| _{L^2(\mathbb{R}^N)}$,
\begin{align}  \label{ide.s}
S(u_{\lambda})=\frac{\lambda^2}{2} \left(\| D_1 u\|  _{L^2(\mathbb{R}^2)}^2 + \| D_2 u\| _{L^2(\mathbb{R}^2)}^2\right) + \frac 12 \| u\| _{L^{2}(\mathbb{R}^2)}^2 -\frac{\lambda^{p-1}}{p+1} \| u\| _{L^{p+1}(\mathbb{R}^2)}^{p+1}
\end{align}
and
\begin{align}\label{ide.k}
K(u_{\lambda})=\lambda^2\left(\| D_1 u\|  _{L^2(\mathbb{R}^2)}^2 +  \| D_2 u\| _{L^2(\mathbb{R}^2)}^2\right) -\frac{\lambda^{p-1}(p-1)}{p+1} \| u\| _{L^{p+1}(\mathbb{R}^2)}^{p+1}.
\end{align}
In addition, we see that
\begin{align} \label{ide.sd1}
\frac{d }{d \lambda} S(u_{\lambda}){\big\vert_{\lambda=1}}=K(u).
\end{align}

\subsection{Global existence of solutions} We start with showing the global existence of solutions to the Cauchy problem for \eqref{CSS} with initial data in $\mathcal{K}^+$.

\begin{lem} \label{lem.gl}
Let $p>3$, then we have that $d>0$, where
\begin{align*}
d= \inf \left\{ S(u): u \in H^1_{rad}(\mathbb{R}^2)  \backslash \{0\}, \, K(u)=0\right\}.
\end{align*}
\end{lem}
\begin{proof}
From \eqref{MGN}, there holds that
\begin{align} \label{inequ}
\| u\| _{L^{p+1}(\mathbb{R}^2)}  \lesssim \left(\| D_1 u\| ^2_{L^2(\mathbb{R}^2)} + \| D_2 u\| ^2_{L^2(\mathbb{R}^2)}\right)^{{\frac{p-1}{2(p+1)}}} \| u\| _{L^2(\mathbb{R}^2)}^{\frac{2}{p+1}}.
\end{align}
Thus, for any $u \in H^1_{rad}(\mathbb{R}^2) \backslash \{0\}$ with $K(u)=0$, we get that
\begin{align*}
\| D_1 u\| ^2_{L^2(\mathbb{R}^2)} + \| D_2 u\| ^2_{L^2(\mathbb{R}^2)} \lesssim \left(\| D_1 u\| ^2_{L^2(\mathbb{R}^2)} + \| D_2 u\| ^2_{L^2(\mathbb{R}^2)}\right)^{\frac{p-1}{2}} \| u\| _{L^2(\mathbb{R}^2)}^{2},
\end{align*}
which gives that
\begin{align} \label{ld}
\left(\| D_1 u\| ^2_{L^2(\mathbb{R}^2)} + \| D_2 u\| ^2_{L^2(\mathbb{R}^2)}\right)^{\frac{p-3}{2}} \| u\| _{L^2(\mathbb{R}^2)}^{2} \gtrsim 1.
\end{align}
On the other hand, for any $u \in H^1_{rad}(\mathbb{R}^2) \backslash \{0\}$ with $K(u)=0$, we know that
\begin{align} \label{su}
\begin{split}
S(u)&=S(u) - \frac{1}{p-1} K(u) \\
&= \frac {p-3}{2(p-1)} \left(\| D_1 u\| ^2_{L^2(\mathbb{R}^2)} + \| D_2 u\| ^2_{L^2(\mathbb{R}^2)}\right) + \frac 12 \| u\| ^2_{L^2(\mathbb{R}^2)} .
\end{split}
\end{align}
Thanks to $p>3$, then it yields from \eqref{ld} and \eqref{su} that $d>0$, and the proof is completed.
\end{proof}

\begin{proof}[Proof of Theorem \ref{thm.scattering1} $\textnormal{i)}$] Let us recall that
$$
\mathcal{K}^+= \left\{u  \in H^1_{rad}(\mathbb{R}^2) : S(u) <d, \, K(u)  >0\right\}.
$$
We first prove that $\mathcal{K}^+$ is invariant under the flow of the Cauchy problem for \eqref{CSS}, i.e. if $u_0 \in \mathcal{K}^+$, then $u(t) \in \mathcal{K}^+$ for any $t \in [0, T_{\textnormal{max}})$. We suppose by contradiction that there exists $t_0  \in (0, T_{\textnormal{max}})$ such that $u(t_0) \notin \mathcal{K}^+$. Since $u_0 \in \mathcal{K}^+$, from the conservation laws, then we know that $S(u(t_0))=S(u_0)<d$. Therefore, there holds that $K(u(t_0)) \leq 0$, because of $u(t_0) \notin \mathcal{K}^+$. Since $u \in C([0, T_{\textnormal{max}}), H^1_{rad}(\mathbb{R}^2))$ and $K(u_0)>0$, then there exists $ t^* \in (0, t_0]$ such that $K(u(t^*))=0$, from which we get that $d \leq S(u(t^*))$. By using again the conservation laws, then we obtain that $d \leq S(u_0)$. This is impossible, because of $u_0 \in \mathcal{K}^+$. Hence we obtain that $\mathcal{K}^+$ is invariant. As a consequence, we have that $K(u(t))>0$ for any $t \in [0, T_{\textnormal{max}})$. This indicates that, for any $t \in [0, T_{\textnormal{max}})$,
\begin{align} \label{dinequ}
\begin{split}
d > S(u(t)) & > S(u(t)) -\frac{1}{p-1} K(u(t)) \\
&= \frac {p-3}{2(p-1)} \left(\| D_1 u(t)\| ^2_{L^2(\mathbb{R}^2)} + \| D_2 u(t)\| ^2_{L^2(\mathbb{R}^2)}\right) + \frac 12 \| u(t)\| ^2_{L^2(\mathbb{R}^2)},
\end{split}
\end{align}
Since $p>3$, then it follows from \eqref{dinequ} that
\begin{align} \label{bounded}
\| D_1 u(t)\| ^2_{L^2(\mathbb{R}^2)}+ \| D_2 u(t)\| ^2_{L^2(\mathbb{R}^2)} \, dx \lesssim 1 \quad \mbox{for any} \,\,\, t \in [0, T_{\textnormal{max}}).
\end{align}
Taking into account \eqref{MGN} and conversation of the mass, then we derive from \eqref{bounded} that $ \| u(t)\| _{L^q(\mathbb{R}^2)} \lesssim 1$ for any $t \in [0, T_{\textnormal{max}})$.
In virtue of Cauchy-Schwarz's inequality, \eqref{bounded} and Lemma \ref{lem.ineq}, then we derive that
\begin{align*}
\| \nabla u(t)\| _{L^2(\mathbb{R}^2)}^2 & = \| \nabla u(t) + \textnormal{i} {A} u(t) - \textnormal{i} {A}u(t)\| _{L^2(\mathbb{R}^2)}^2  \\
& \lesssim  \| D_1 u(t)\|  _{L^2(\mathbb{R}^2)}^2 +  \| D_2 u(t)\| _{L^2(\mathbb{R}^2)}^2 +  \| {A} u(t)\| _{L^2(\mathbb{R}^2)}^2 \\
& \lesssim 1,
\end{align*}
where $A=(A_1, A_2)$. According to Theorem \ref{thm.wellposedness}, then we deduce that $u$ exists globally in time, i.e. $T_{\textnormal{max}}=+\infty$. This completes the proof.
\end{proof}

\subsection{Scattering of solutions} Let us now investigate scattering of solutions to the Cauchy problem for \eqref{CSS} with initial data in $\mathcal{K}^+$. To do so, we first need to establish the following crucial lemma.

\begin{lem} \label{lem.below}
Let $p>3$ and let $u \in C([0, + \infty), H^1_{rad}(\mathbb{R}^2)) $ be the solution to the Cauchy problem for \eqref{CSS} with initial datum $u_0 \in \mathcal{K}^+$, then
$$
K(u(t)) \gtrsim \| D_1 u(t)\|  _{L^2(\mathbb{R}^2)}^2 + \| D_2 u(t)\| _{L^2(\mathbb{R}^2)}^2 \quad \mbox{for any} \,\,\, t \in [0, + \infty).
$$
\end{lem}
\begin{proof}
From \eqref{ide.s} and \eqref{ide.k}, we first get that
\begin{align} \label{ind.sd}
\begin{split}
\frac{d^2 }{d \lambda^2} S((u(t))_{\lambda})&=\| D_1 u(t)\|  _{L^2(\mathbb{R}^2)}^2 + \| D_2 u(t)\| _{L^2(\mathbb{R}^2)}^2 -\frac{(p-1)(p-2)}{p+1} \lambda^{p-3}\| u(t)\| _{L^{p+1}(\mathbb{R}^2)}^{p+1} \\
&=-\frac {1} {\lambda^{2} }K((u(t))_{\lambda}) + \frac {2} {\lambda^{2}} \left(K((u(t))_{\lambda}) -\frac{(p-1)(p-3)}{2(p+1)} \| (u(t))_{\lambda}\| _{L^{p+1}(\mathbb{R}^2)}^{p+1} \right).
\end{split}
\end{align}
For any $t \geq 0$, we now define
\begin{align*}
T_1:=\left\{t \in [0, \infty) : K(u(t)) \geq \frac{(p-1)(p-3)}{2(p+1)} \| u(t)\| _{L^{p+1}(\mathbb{R}^2)}^{p+1}\right\}
\end{align*}
and $T_2:=[0, \infty)\backslash T_1$.  If $t \in T_1$, then
\begin{align*}
K(u(t)) &\geq \frac{(p-1)(p-3)}{2(p+1)} \| u(t)\| _{L^{p+1}(\mathbb{R}^2)}^{p+1} \\
&= -\frac{p-3}{2} K(u(t)) + \frac{p-3}{2}\left(\| D_1 u(t)\|  _{L^2(\mathbb{R}^2)}^2 + \| D_2 u(t)\| _{L^2(\mathbb{R}^2)}^2\right).
\end{align*}
This yields that
\begin{align} \label{k1}
K(u(t))  \geq \frac{p-3}{p-1}\left(\| D_1 u(t)\|  _{L^2(\mathbb{R}^2)}^2 + \| D_2 u(t)\| _{L^2(\mathbb{R}^2)}^2\right) \quad \mbox{for any} \,\, t \in T_1.
\end{align}
If $t \in T_2$, then
\begin{align} \label{ineq.ku}
K(u(t)) < \frac{(p-1)(p-3)}{2(p+1)} \| u(t)\| _{L^{p+1}(\mathbb{R}^2)}^{p+1}.
\end{align}
In this case, for simplicity we set $v=u(t)$ with $t \in T_2$. Then \eqref{ineq.ku} leads to
\begin{align} \label{ineq.u}
\| D_1 v\|  _{L^2(\mathbb{R}^2)}^2 + \| D_2 v\| _{L^2(\mathbb{R}^2)}^2 <  \frac{(p-1)^2}{2(p+1)} \| v\| _{L^{p+1}(\mathbb{R}^2)}^{p+1}.
\end{align}
Since $u_0 \in \mathcal{K}^+$ and $\mathcal{K}^+$ is invariant, then we know that $v \in \mathcal{K}^+$, which means that $K(v) >0$. Hence it is not hard to deduce that there exists $\lambda_v>1$ such that $K(v_{\lambda_v})=0$, namely,
\begin{align} \label{ide.lv}
\| D_1 v\|  _{L^2(\mathbb{R}^2)}^2 + \| D_2 v\| _{L^2(\mathbb{R}^2)}^2=\frac{p-1}{p+1} {\lambda_v}^{p-3} \| v\| _{L^{p+1}(\mathbb{R}^2)}^{p+1}.
\end{align}
Combining \eqref{ineq.u} with \eqref{ide.lv}, then we derive that
\begin{align} \label{ineq.lambda}
\lambda_v < \left(\frac{p-1}{2}\right)^{\frac{1}{p-3}}.
\end{align}
Due to $p>3$, then
\begin{align*}
\frac{d}{d\lambda} \left(\frac{K(v_{\lambda})}{\lambda^2}-\frac{ \| v_{\lambda}\| _{L^{p+1}(\mathbb{R}^2)}^{p+1}}{\lambda^2}\right)&=-\frac{(p-1)(p-3)}{p+1} \lambda^{p-4}\| v\| _{L^{p+1}(\mathbb{R}^2)}^{p+1}-(p-3) \lambda^{p-4}\| v\| _{L^{p+1}(\mathbb{R}^2)}^{p+1}<0.
\end{align*}
Therefore, we can conclude from \eqref{ineq.ku} that, for any $\lambda \geq 1$,
$$
\frac {2} {\lambda^{2}} \left(K(v_{\lambda}) -\frac{(p-1)(p-3)}{2(p+1)} \| v_{\lambda}\| _{L^{p+1}(\mathbb{R}^2)}^{p+1} \right)  \leq {2} \left(K(v) -\frac{(p-1)(p-3)}{2(p+1)} \| v\| _{L^{p+1}(\mathbb{R}^2)}^{p+1} \right)<0,
$$
Coming back to \eqref{ind.sd}, then we get that
\begin{align} \label{ineq.d2}
\frac{d^2 }{d \lambda^2}  S(v_{\lambda}) < 0 \quad \mbox{for any} \,\,\, \lambda \geq 1.
\end{align}
This together with \eqref{ineq.lambda} shows that
\begin{align*}
\left(\left(\frac{p-1}{2}\right)^{\frac{1}{p-3}}-1\right) K(v) &\geq (\lambda_v-1) \frac{d}{d \lambda} S(v_{\lambda})\big\vert_{\lambda=1}  \geq S(v_{\lambda_v})-S(v) \geq d -S(u_0),
\end{align*}
where the last inequality benefits from the fact that $K(v_{\lambda_v})=0$ and the conservation laws. As a result, we have that
\begin{align} \label{k2}
K(u(t))=K(v) \gtrsim d -S(u_0) \quad \mbox{for any} \,\, t \in T_2.
\end{align}
Consequently, combining \eqref{k1} and \eqref{k2}, we arrive at
\begin{align} \label{ineq.b1}
K(u(t)) \gtrsim \min \left\{\| D_1 u(t)\|  _{L^2(\mathbb{R}^2)}^2 + \| D_2 u(t)\| _{L^2(\mathbb{R}^2)}^2, \, d-S(u_0)\right\} \quad \mbox{for any} \,\,\, t \in [0, \infty).
\end{align}
To complete the proof, we only need to show that
\begin{align}\label{eq.fp3}
d-S(u_0) \gtrsim \| D_1 u(t)\|  _{L^2(\mathbb{R}^2)}^2 + \| D_2 u(t)\| _{L^2(\mathbb{R}^2)}^2 \quad \mbox{for any} \,\,t \in [0, \infty).
\end{align}
Since $K(u(t))>0$, then
\begin{align} \label{ineq.b2}
\| D_1 u(t)\|  _{L^2(\mathbb{R}^2)}^2 + \| D_2 u(t)\| _{L^2(\mathbb{R}^2)}^2 < \alpha S(u(t)) \quad \mbox{for any} \,\, t \in [0, + \infty) , \, \, \alpha \geq \frac{2(p-1)}{p-3}.
\end{align}
Indeed, if $\alpha \geq \frac{2(p-1)}{p-3}$, then
$$
\frac{2 \alpha}{(\alpha-2)(p+1)} \leq \frac{p-1}{p+1}.
$$
This implies that
\begin{align*}
&\alpha S(u(t)) - \left(\| D_1 u(t)\|  _{L^2(\mathbb{R}^2)}^2 + \| D_2 u(t)\| _{L^2(\mathbb{R}^2)}^2\right) \\
&= \frac{\alpha-2}{2}\left(\| D_1 u(t)\|  _{L^2(\mathbb{R}^2)}^2 + \| D_2 u(t)\| _{L^2(\mathbb{R}^2)}^2\right)-\frac{\alpha}{p+1}\| u(t)\| _{L^{p+1}(\mathbb{R}^2)}^{p+1} \\
&=\frac{\alpha-2}{2} \left(\| D_1 u(t)\|  _{L^2(\mathbb{R}^2)}^2 + \| D_2 u(t)\| _{L^2(\mathbb{R}^2)}^2- \frac{2 \alpha}{(\alpha-2)(p+1)}\| u(t)\| _{L^{p+1}(\mathbb{R}^2)}^{p+1}\right)>0,
\end{align*}
from which \eqref{ineq.b2} necessarily follows. Applying again the fact that $K(u(t)) > 0$ for ant $t \geq 0$, then it yields from \eqref{ineq.b2} that
\begin{align*}
\| D_1 u(t)\|  _{L^2(\mathbb{R}^2)}^2 + \| D_2 u(t)\| _{L^2(\mathbb{R}^2)}^2 < \alpha S(u(t)) =\alpha S(u_0) =C_p (d-S(u_0)) \quad \mbox{for any} \,\,t \in [0, \infty),
\end{align*}
where $C_p>0$ is defined by
$$
C_p:=\frac{\alpha S(u_0)} {d-S(u_0)}.
$$
Accordingly, we get that
$$
\| D_1 u(t)\|  _{L^2(\mathbb{R}^2)}^2 + \| D_2 u(t)\| _{L^2(\mathbb{R}^2)}^2  \lesssim d-S(u_0) \quad \mbox{for any} \,\,t \in [0, \infty).
$$
This gives \eqref{eq.fp3} and completes the proof.
\end{proof}

\begin{lem} \label{lem.stest}
Let $p>3$ and let $u\in C([0, + \infty), H^1_{rad}(\mathbb{R}^2)) $ be the solution to the Cauchy problem for \eqref{CSS} with initial datum $u_0 \in \mathcal{K}^+$, then, for any $T>0$,
\begin{align*}
& \int_{0}^T \int_{\mathbb{R}^2} \mid u\mid ^{p+1} \, dx dt +\int_{0}^T \int_{\mathbb{R}^2} \left(\frac{A_{\theta}(\mid u\mid ^2)}{r}\right)^2 \mid u\mid ^2  \, dxdt  \\ \nonumber & + \int_{0}^T \int_{\mathbb{R}^2} A_0(\mid u\mid ^2)\mid u\mid ^2 \, dx dt + \int_{0}^T \int_{\mid x\mid  \leq T^{\alpha}} \mid \nabla u\mid ^2 \, dx dt  \lesssim T^{\alpha},
\end{align*}
where $\alpha= 1/(1+\sigma)$ and $\sigma=\min \left\{2, \frac{p-1}{2}\right\}$.
\end{lem}
\begin{proof}
Let $\varphi \in C^{\infty}(\mathbb{R}^+, [0, 1])$ be such that $\varphi(r)=1$ for any $ 0 \leq r \leq 1$, $\varphi(r)=0$ for any $r \geq 2$ and $\varphi'(r) \leq 0$ for any $r \geq 0$.
Define
$$
\psi(r)= \frac 1 r \int_{0}^r \varphi (\rho) \, d \rho.
$$
Notice first that
\begin{align} \label{ide.dpsi}
\psi'(r)= \frac 1 r \varphi(r) -  \frac {1} {r^2} \int_{0}^r \varphi (\rho) \, d \rho,
\end{align}
then
\begin{align} \label{ide.phi}
r\psi'(r) + \psi(r) = \varphi(r) \quad \mbox{for any} \,\,\, r \geq 0.
\end{align}
In addition, there holds that $\psi'(r)=0$ for any $0 \leq r <1$ and
\begin{align} \label{psidecr}
\psi'(r) \leq 0 \quad \mbox{for any} \,\,\, r \geq 1.
\end{align}
In view of \eqref{ide.dpsi}, one can see that
\begin{align}\label{psiest}
\mid \psi'(r)\mid   \lesssim \frac {1}{r^2} \quad \mbox{for any} \,\,\, r \geq 1.
\end{align}
For $R>0$, we now define
$$
\psi_R(x)=\psi\left(\frac{\mid x\mid }{R}\right) \quad \mbox{for any} \, \, x \in \mathbb{R}^2
$$
and
$$
M_{\psi_R}(u)= \textnormal{Im} \int_{\mathbb{R}^2} \psi_R \left(x_1 D_1 u + x_2 D_2 u\right) \overline{u}\, dx.
$$
Applying \cite[Lemma 6.2]{Gou}, then we derive that
\begin{align} \label{ide.dem}
\begin{split}
\frac{d}{dt}M_{\psi_R}[u]=&2 \int_{\mathbb{R}^2} \mid D_1 u\mid ^2 \left( \psi\left(\frac{\mid x\mid }{R}\right) + \psi'\left(\frac{\mid x\mid }{R}\right) \frac{x_1^2}{R\mid x\mid } \right) \, dx \\
& \quad  +  4 \, \textnormal{Re} \int_{\mathbb{R}^2} \overline{D_1 u} \, D_2 u \, \psi'\left(\frac{\mid x\mid }{R}\right)\frac{x_1 x_2}{\mid x\mid ^2} \, dx \\
&\quad + 2 \int_{\mathbb{R}^2} \mid D_2 u\mid ^2 \left( \psi\left(\frac{\mid x\mid }{R}\right) + \psi'\left(\frac{\mid x\mid }{R}\right)\frac{x_2^2}{R\mid x\mid } \right) \, dx \\
& \quad - \frac{p-1}{p+1} \int_{\mathbb{R}^2} \mid u\mid ^{p+1} \left(\varphi\left(\frac{\mid x\mid }{R}\right)+ \psi\left(\frac{\mid x\mid }{R}\right)\right) \, dx + \mathcal{O}\left(\frac {1}{R^{2}}\right).
\end{split}
\end{align}
In light of \eqref{ide.phi}, we know that
\begin{align*}
\psi\left(\frac{\mid x\mid }{R}\right) + \psi'\left(\frac{\mid x\mid }{R}\right) \frac{x_1^2}{R\mid x\mid } &= \varphi \left(\frac{\mid x\mid }{R}\right) - \psi'\left(\frac{\mid x\mid }{R}\right) \frac{x_2^2}{R\mid x\mid }
\end{align*}
and
\begin{align*}
\psi\left(\frac{\mid x\mid }{R}\right) + \psi'\left(\frac{\mid x\mid }{R}\right) \frac{x_2^2}{R\mid x\mid } &= \varphi \left(\frac{\mid x\mid }{R}\right) - \psi'\left(\frac{\mid x\mid }{R}\right) \frac{x_1^2}{R\mid x\mid }.
\end{align*}
Therefore, from \eqref{ide.dem}, there holds that
\begin{align*}
\frac{d}{dt}M_{\psi_R}[u] &=  2 \int_{\mathbb{R}^2} \varphi \left(\frac{\mid x\mid }{R}\right) \left(\mid D_1 u\mid ^2 + \mid D_2 u\mid  -\frac{p-1}{p
+1} \mid u\mid ^{p+1}\right) \, dx \\
& \quad - 2\int_{\mathbb{R}^2} \frac{1}{R\mid x\mid } \psi'\left(\frac{\mid x\mid }{R}\right) \mid x_2 D_1 u + x_1 D_2 u\mid ^2 \, dx \\
& \quad - \frac{p-1}{p+1} \int_{\mathbb{R}^2} \mid u\mid ^{p+1} \left(\psi\left(\frac{\mid x\mid }{R}\right)-\varphi\left(\frac{\mid x\mid }{R}\right)\right) \, dx + \mathcal{O}\left(\frac{1}{R^2}\right).
\end{align*}
Note that $\psi(r)=\varphi(r)$ for any $0 \leq r \leq 1$, then it yields from the Strauss inequality \eqref{Strauss}, \eqref{ide.phi} and \eqref{psiest} that
\begin{align*}
\int_{\mathbb{R}^2} \mid u\mid ^{p+1} \left(\varphi\left(\frac{\mid x\mid }{R}\right)-\psi\left(\frac{\mid x\mid }{R}\right)\right) \, dx &=\int_{\mathbb{R}^2} \frac{\mid x\mid }{R}\psi'\left(\frac{\mid x\mid }{R}\right)\mid u\mid ^{p+1} \, dx \\
&\lesssim R \int_{\mathbb{R}^2 \backslash B_R(0)} \frac {\mid u\mid ^{p-1}} {\mid x\mid }\mid u\mid ^2\, dx \\
& = \mathcal{O}\left(\frac{1}{R^{\frac{p-1}{2}}}\right).
\end{align*}
On the other hand, from \eqref{psidecr}, we have that $\psi'(r) \leq 0$ for any $r \geq 0$. As a consequence, we obtain that
\begin{align} \label{ineq.m}
\hspace{-1cm}\frac{d}{dt}M_{\psi_R}[u] & \gtrsim  2 \int_{\mathbb{R}^2} \varphi \left(\frac{\mid x\mid }{R}\right) \left(\mid D_1 u\mid ^2 + \mid D_2 u\mid  -\frac{p-1}{p
+1} \mid u\mid ^{p+1}\right) \, dx - \mathcal{O}\left(\frac{1}{R^{\sigma}}\right).
\end{align}
For further discussions, let us now define
$$
\varphi_R(x)=\varphi\left(\frac{\mid x\mid }{R}\right), \quad \phi_R^2(x)=\varphi_R(x)  \quad \mbox{for any} \,\, x \in \mathbb{R}^2.
$$
Next we aim to prove that $\phi_R u \in \mathcal{K}^+$ for $R>0$ sufficiently large. To begin with, making use of the third identity of \eqref{eq.gg3},
we can write
$$
S(\phi_R u)= \frac 12 \int_{\mathbb{R}^2} \mid \nabla (\phi_R u)\mid ^2 + \left(\frac{A_\theta(\mid \phi_R u\mid ^2)}{r}\right)^2\mid \phi_R u\mid ^2  \, dx + \frac 12 \int_{\mathbb{R}^2} \mid \phi_R u\mid ^2 \, dx - \frac{1}{p+1} \int_{\mathbb{R}^2} \mid \phi_R u\mid ^{p+1} \, dx
$$
and
$$
K(\phi_R u)= \int_{\mathbb{R}^2} \mid \nabla (\phi_R u)\mid ^2 + \left(\frac{A_\theta(\mid \phi_R u\mid ^2)}{r} \right)^2\mid \phi_R u\mid ^2  \, dx - \frac{p-1}{p+1} \int_{\mathbb{R}^2} \mid \phi_R u\mid ^{p+1} \, dx.
$$
From the definition of $\phi_R$, we are able to derive that $\| \phi_R u\| _{L^2(\mathbb{R}^2)}^2 \leq \| u\| _{L^2(\mathbb{R}^2)}^2$ and
\begin{align} \label{est.grad}
\begin{split}
\int_{\mathbb{R}^2}\mid \nabla(\phi_R u)\mid ^2 \, dx &= \int_{\mathbb{R}^2} \varphi_R \mid \nabla u\mid ^2 \,dx - \int_{\mathbb{R}^2} \phi_R \Delta \phi_R \mid u\mid ^2 \, dx \\
&= \int_{\mathbb{R}^2} \varphi_R \mid \nabla u\mid ^2 \,dx + \mathcal{O}\left(\frac{1}{R^2}\right) \\
& \leq \int_{\mathbb{R}^2} \mid \nabla u\mid ^2 \,dx + \mathcal{O}\left(\frac{1}{R^2}\right).
\end{split}
\end{align}
Observe that
\begin{align*}
&\left(\frac{A_\theta(\mid \phi_R u\mid ^2)}{r}\right)^2\mid \phi_R u\mid ^2 - \left(\frac{A_\theta(\mid  u\mid ^2)}{r}\right)^2\mid  u\mid ^2 \\
&= \left(\left(\frac{A_\theta(\mid \phi_R u\mid ^2)}{r}\right)^2-\left(\frac{A_\theta(\mid u\mid ^2)}{r} \right)^2\right)\mid \phi_R u\mid ^2
+ \left(\frac{A_\theta(\mid  u\mid ^2)}{r}\right)^2\left(\mid \phi_R u\mid ^2 - \mid u\mid ^2 \right) \\
&= \left(\frac{\left(A_\theta((1+\varphi_R)\mid  u\mid ^2\right) \left(A_\theta((1-\varphi_R)\mid  u\mid ^2)\right)}{r^2}\right)\mid \phi_R u\mid ^2
+ \left(\frac{A_\theta(\mid  u\mid ^2)}{r}\right)^2\left(\varphi_R-1 \right)\mid u\mid ^2.
\end{align*}
Taking into account Lemma \ref{lem.ineq} with $b=0$, $q=\infty$ and $s=2$, we find that
\begin{align*}
 \mid \int_{\mathbb{R}^2 \backslash B_R(0)} \frac{\left(A_\theta((1-\varphi_R)\mid u\mid ^2)\right)^2}{r^2}\mid \phi_R u\mid ^2+ \frac{\left(A_\theta((1+\varphi_R)\mid  u\mid ^2)\right)^2}{r^2}\mid \phi_R u\mid ^2\, dx \mid \leq \mathcal{O}\left(\frac{1}{R^2}\right)
\end{align*}
and
\begin{align*}
 \mid \int_{\mathbb{R}^2 \backslash B_R(0)} \left(\frac{A_\theta(\mid u\mid ^2)}{r}\right)^2\left(\varphi_R-1 \right)\mid u\mid ^2 \, dx  \mid \leq \mathcal{O}\left(\frac{1}{R^2}\right).
\end{align*}
As a result, we get that
\begin{align} \label{est.a}
\mid \int_{\mathbb{R}^2}\left(\frac{A_\theta(\mid \phi_R u\mid ^2)}{r}\right)^2\mid \phi_R u\mid ^2 - \left(\frac{A_\theta(\mid  u\mid ^2)}{r}\right)^2\mid  u\mid ^2 \, dx \mid  \leq \mathcal{O}\left(\frac{1}{R^2}\right).
\end{align}
Moreover, we can show that
\begin{align} \label{est.nt}
\begin{split}
\mid \int_{\mathbb{R}^2} \mid \phi_{R} u\mid ^{p+1} \, dx -\int_{\mathbb{R}^2} \mid u\mid ^{p+1} \, dx \mid  & = \int_{\mathbb{R}^2} \left(1- \mid \phi_{R}\mid ^{p+1}\right) \mid u\mid ^{p+1} \, dx \\
& \leq \int_{\mathbb{R}^2 \backslash B_R(0)}  \mid u\mid ^{p-1}\mid u\mid ^2 \, dx \\
& = \mathcal{O}\left(\frac{1}{R^{\frac{p-1}{2}}}\right).
\end{split}
\end{align}
Note that $u \in \mathcal{K}^+$, by the conservation laws, then there exists a constant $\delta>0$ such that $S(u)<(1-\delta) d$ for any $t \in [0, + \infty)$.  In view of \eqref{est.grad}, \eqref{est.a} and \eqref{est.nt}, then we know that $S(\phi_R u) <d $ for $R>0$ sufficiently large. Relying on Lemmas \ref{lem.gl} and \ref{lem.below} and using a similar way as before, we are able to prove that $K(\phi_R u) >0$ for $ R>0 $ sufficiently large. This readily indicates that $\phi_R u \in \mathcal{K}^+$ for $R>0$ sufficiently large. Arguing as the proof of Lemma \ref{lem.below}, then we have that, for $R>0$ sufficiently large,
\begin{align} \label{est.phiu}
K(\phi_R u) \gtrsim \int_{\mathbb{R}^2} \mid D_1(\phi_R u)\mid ^2 + \mid D_2 (\phi_R u)\mid ^2 \, dx \geq \frac{p-1}{p+1} \int_{\mathbb{R}^2}\mid \phi_R u\mid ^{p+1} \, dx.
\end{align}
Further we can similarly infer that
$$
\int_{\mathbb{R}^2}  \varphi_R \left(\mid D_1 u\mid ^2 + \mid D_2 u\mid ^2 -\frac{p-1}{p
+1} \mid u\mid ^{p+1}\right) \, dx + \mathcal{O}\left(\frac{1}{R^{\sigma}}\right) \geq K(\phi_R u).
$$
This together with \eqref{est.phiu} yields that
\begin{align*}
\int_{\mathbb{R}^2}  \varphi_R \left(\mid D_1 u\mid ^2 + \mid D_2 u\mid  -\frac{p-1}{p
+1} \mid u\mid ^{p+1}\right) \, dx &\geq \int_{\mathbb{R}^2}\mid \phi_R u\mid ^{p+1} \, dx  - \mathcal{O}\left(\frac{1}{R^{\sigma}}\right) \\
& \geq \int_{B_R(0)}\mid u\mid ^{p+1} \, dx  - \mathcal{O}\left(\frac{1}{R^{\sigma}}\right) .
\end{align*}
By means of \eqref{ineq.m}, we now get that
\begin{align} \label{ineq.dmb}
\frac{d}{dt}M_{\psi_R}[u] \gtrsim \int_{B_R(0)}\mid u\mid ^{p+1} \, dx  - \mathcal{O}\left(\frac{1}{R^{\sigma}}\right).
\end{align}
Notice that $\mid M_{\psi_R}(u)\mid  \lesssim R$ for any $t \in [0, +\infty)$, then it follows from \eqref{ineq.dmb} that
\begin{align} \label{ineq.br}
\int_{0}^T \int_{B_R(0)} \mid u\mid ^{p+1} \, dx dt \lesssim R + \frac{T}{R^{\sigma}}.
\end{align}
On the other hand, by the Strauss inequality \eqref{Strauss}, we know that
\begin{align}\label{ineq.br1}
\int_{0}^T \int_{\mathbb{R}^2 \backslash B_R(0)} \mid u\mid ^{p+1} \, dx dt=\int_{0}^T \int_{\mathbb{R}^2 \backslash B_R(0)} \mid u\mid ^{p-1}\mid u\mid ^2 \, dx dt \lesssim \frac{T}{R^{\frac{p-1}{2}}} \leq \frac{T}{R^{\sigma}}.
\end{align}
Consequently, by invoking \eqref{ineq.br}, \eqref{ineq.br1} and
taking $R=T^{{1}/{(1+\sigma)}}$, we derive that
$$
\int_{0}^T \int_{\mathbb{R}^2} \mid u\mid ^{p+1} \, dx \lesssim T^{\frac{1}{1+\sigma}}.
$$
Make using of \eqref{est.phiu}, we can also derive that
\begin{align} \label{ineq.at1}
\int_{0}^T \int_{B_R(0)}\left(\frac{A_{\theta}(\mid u\mid ^2)}{r}\right)^2 \mid u\mid ^2 + \mid \nabla u\mid ^2 \, dx dt  \lesssim R + \frac{T}{R^{\sigma}}.
\end{align}
By applying Lemma \ref{lem.ineq} with $b=0$, $q=\infty$ and $s=2$, we see that
\begin{align} \label{ineq.at2}
\int_{0}^T \int_{\mathbb{R}^2 \backslash B_R(0)} \left(\frac{A_{\theta}(\mid u\mid ^2)}{r}\right)^2 \mid u\mid ^2 \, dx dt \lesssim \frac{T}{R^2}.
\end{align}
Combining \eqref{ineq.at1} and \eqref{ineq.at2}, then we get that
\begin{align}\label{ineq.at}
\int_{0}^T \int_{\mathbb{R}^2} \left(\frac{A_{\theta}(\mid u\mid ^2)}{r}\right)^2 \mid u\mid ^2 \, dx dt \lesssim T^{\frac{1}{1+\sigma}}.
\end{align}
Notice that
$$
\int_{\mathbb{R}^2}A_0(\mid u\mid ^2)\mid u\mid ^2 \, dx =2\pi \int_{0}^{\infty}A_0(\mid u\mid ^2)\mid u\mid ^2 r \,dr=-2\pi \int_{0}^{\infty}\left( \int_{r}^{\infty}\frac{A_{\theta}(\mid u\mid ^2)}{\rho}\mid u\mid ^2(\rho) \, d \rho\right) \mid u\mid ^2(r) r\,dr
$$
and
$$
\mid u\mid ^2 r=\partial_r \left(\int_{0}^{r}\mid u\mid ^2 \rho \,d \rho\right)=-2 \partial_r \left(A_{\theta}(\mid u\mid ^2)\right).
$$
Therefore, we find that
$$
\int_{\mathbb{R}^2}A_0(\mid u\mid ^2)\mid u\mid ^2 \, dx=4 \pi \int_{0}^{\infty} \left(\frac{A_{\theta}(\mid u\mid ^2)}{r}\right)^2\mid u\mid ^2 r \, dr=2\int_{\mathbb{R}^2}\left(\frac{A_{\theta}(\mid u\mid ^2)}{r}\right)^2\mid u\mid ^2 \, dx.
$$
Taking advantage of \eqref{ineq.at}, then we obtain that
$$
\int_{0}^T \int_{\mathbb{R}^2} A_0(\mid u\mid ^2)\mid u\mid ^2 \, dx dt \lesssim T^{\frac{1}{1+\sigma}}.
$$
Finally, by applying \eqref{ineq.at1} and taking $R = T^{1/(1+\sigma)}$, we can derive that
$$
\int_{0}^T \int_{\mid x\mid  \leq T^{1/(1+\sigma)}} \mid \nabla u\mid ^2 \, dx dt  \lesssim T^{\frac {1}{1+\sigma}}.
$$
This completes the proof.
\end{proof}

\begin{cor} \label{cor.sm}
Let $p>3$ and let $u\in C([0, + \infty), H^1_{rad}(\mathbb{R}^2)) $ be the solution to the Cauchy problem for \eqref{CSS} with initial datum $u_0 \in \mathcal{K}^+$, then, for any $\varepsilon >0$,
there exist $\delta=\delta\left(\varepsilon \right) \in (0,\varepsilon)$, $T=T(\varepsilon)>1/\varepsilon$ and $t_{0} \in\left[{T}/{4}, {T}/{2}\right]$ such that $t_1 = t_0+ \delta T^{1-\alpha} < T/2$ and
\begin{align} \label{ineq.ct0}
\begin{split}
& \int_{t_0}^{t_1} \int_{\mathbb{R}^2} \mid u\mid ^{p+1} \, dx dt +\int_{t_0}^{t_1} \int_{\mathbb{R}^2} \left(\frac{A_{\theta}(\mid u\mid ^2)}{r}\right)^2 \mid u\mid ^2  \, dxdt  \\
& + \int_{t_0}^{t_1} \int_{\mathbb{R}^2} A_0(\mid u\mid ^2)\mid u\mid ^2 \, dx dt + \int_{t_0}^{t_1} \int_{\mid x\mid  \leq T^{\alpha}} \mid \nabla u\mid ^2 \, dx dt  \lesssim \varepsilon.
\end{split}
\end{align}
\end{cor}
\begin{proof}
 From Lemma \ref{lem.stest}, we know that, for any $T>0$ large enough,
\begin{align}  \label{ineq.ct1}
& \int_{0}^T \int_{\mathbb{R}^2} \mid u\mid ^{p+1} \, dx dt +\int_{0}^T \int_{\mathbb{R}^2} \left(\frac{A_{\theta}(\mid u\mid ^2)}{r}\right)^2 \mid u\mid ^2  \, dxdt  \\ \nonumber & + \int_{0}^T \int_{\mathbb{R}^2} A_0(\mid u\mid ^2)\mid u\mid ^2 \, dx dt + \int_{0}^T \int_{\mid x\mid  \leq T^{\alpha}} \mid \nabla u\mid ^2 \, dx dt  \lesssim T^{\alpha}.
\end{align}
For any $\delta \in (0,\varepsilon)$, the interval $\left[{T}/{4}, {T}/{2}\right]$ can be covered by $\sim \delta^{-1} T^{\alpha}$ disjoint intervals of length $\delta T^{1-\alpha},$ then \eqref{ineq.ct1} indicates that there exists some interval $[t_{0}, t_1]\subset \left[{T}/{4}, {T}/{2}\right]$, i.e.
$$
\frac{T}{4} < t_0 < t_1 =  t_{0}+\delta T^{1-\alpha} < \frac{T}{2},
$$
obeying \eqref{ineq.ct0}. Thus the proof is completed.
\end{proof}

\begin{rem} \label{rem.des}
It is important to mention that $T=T(\varepsilon)> 1/\varepsilon$ can be chosen sufficiently large after the choice of $\delta=\delta(\varepsilon)\in (0,\varepsilon)$ such that, for any integer $N \geq 1$, there holds that
\begin{align*}
\max(\varepsilon, \delta)= \varepsilon \lesssim \frac{1}{\ln^N T}.
\end{align*}
Indeed, we can take $1/\varepsilon <T \lesssim e^{1/(\varepsilon^{1/N})}$ for $\varepsilon>0$ small enough.
\end{rem}

We are now ready to prove the scattering of solutions.

\begin{proof}[Proof of Theorem \ref{thm.scattering1} $\textnormal{i)}$]
Let $\varepsilon>0$ be a small parameter (depending on $\| u_{0}\| _{H^{1}(\mathbb{R}^2)}$) to be chosen sufficiently small below.
To achieve scattering, we only need to prove that there exists a constant $t_1= t_1(\varepsilon)>0$ sufficiently large such that
for any $s \in (0,1)$ close to $1$ and for any couple $(1/q,1/r)$ in the triangle
$$
\mathfrak{T}_s:= \left\{\left(\frac{1}{q}, \frac{1}{r}\right) : \frac{1-s}{2}\leq \frac{1}{q}+\frac{1}{r} \leq \frac{1}{2}, \ 2 < q,r < \infty \right\},
$$
there holds that
\begin{align}\label{eq.key1m}
\left\| u\right\| _{L^{q}\left([t_1,\infty); L^r(\mathbb{R}^{2})\right)}    \lesssim o(1),
\end{align}
where $o(1) \to 0$ as $\varepsilon \to 0$.

Let us now explain the plan of the proof of \eqref{eq.key1m}. By using Sobolev embedding inequality and Strichartz estimate, we first have that
\[
\left\| e^{\textnormal{i} t \Delta} u_{0}\right\| _{L^{q}\left([0,\infty); L^r(\mathbb{R}^{2})\right)}
\lesssim\left\| u_{0}\right\| _{H^{1}\left(\mathbb{R}^{2}\right)}.
\]
It then follows that there exists a constant $t_1=t_1(\varepsilon)>0$ sufficiently large such that
\begin{align}\label{eq.sc1}
 \left\| e^{\textnormal{i} t \Delta} u_{0}\right\| _{L^{q}\left([t_1,\infty); L^r(\mathbb{R}^{2})\right)} \lesssim o(1).
\end{align}
At this point, to deduce \eqref{eq.key1m}, it suffices to show that
\begin{align}
\begin{aligned}
\left\| \int_{0}^{t_{1}} e^{\textnormal{i} (t-\tau) \Delta} \Lambda (u)(\tau) \, d \tau\right\| _{L^{q}\left([t_1,\infty); L^r(\mathbb{R}^{2})\right)}   \lesssim  o(1).
\end{aligned}
\label{eq.4.4}
\end{align}
Let us now assume that \eqref{eq.4.4} is valid for the moment and use the Duhamel formula to write
\begin{align*}
e^{\textnormal{i} \left(t-t_{1}\right) \Delta} u\left(t_{1}\right)=e^{\textnormal{i} t \Delta} u_{0}+\textnormal{i} \int_{0}^{t_{1}} e^{\textnormal{i}(t-\tau) \Delta} \Lambda (u)(\tau) \,d \tau.
\end{align*}
Choosing $t_1>0$ sufficiently large depending on $\varepsilon$ and utilizing \eqref{eq.sc1} and \eqref{eq.4.4}, then we deduce that
\begin{align}\label{e.sid1}
\left\| e^{\textnormal{i} \left(t-t_{1}\right) \Delta} u\left(t_{1}\right)\right\| _{L^{q}\left([t_1,\infty); L^r(\mathbb{R}^{2})\right)}
\lesssim o(1).
\end{align}
Observe now that $u$ satisfies the following integral equation,
$$
u(t) = e^{\mathrm{i} (t-t_1) \Delta} u(t_1) +\mathrm{i} \int_{t_1}^t e^{\mathrm{i} (t-\tau)\Delta} \Lambda(u(\tau)) \,ds.
$$
From Strichartz estimates, uniform boundedness of $H^1$-norm of the solution and \eqref{e.sid1},
then we are able to obtain the estimate \eqref{eq.key1m}, see the Appendix for the proof.

In what follows, we are going to prove that \eqref{eq.4.4} holds true.
In view of Corollary \ref{cor.sm}, we first know that
\begin{align} \label{est.small}
\int_{t_{0}}^{t_1} \int_{\mathbb{R}^2} \mid u(t, x)\mid ^{p+1} \,d x d t \lesssim \varepsilon,
\end{align}
where $t_{1}=t_{0}+\delta T^{1-\alpha}$ and $t_0\in[T/4, T/2]$. Notice that $t_{0} \leq \frac{T}{2}$ and $\alpha>0$, then $t_{1}<T$ for $T>0$ sufficiently large. To establish \eqref{eq.4.4}, we shall consider separately the integral on $\left[0, t_{0}\right]$ and $\left[t_{0}, t_{1}\right]$. Let us first treat the integral on $\left[0, t_{0}\right]$. For any $t \geq t_1$, by the definition of $\Lambda_1$ and Lemma \ref{lem.bf1}, we get that
$$
\mid \Lambda_1(t,r)\mid   \lesssim \frac{\mid u(t,r)\mid }{1+r^2}.
$$
Applying standard dispersive estimate, then we have that
\begin{align} \label{est.lambda1}
\begin{split}
\left\| \int_{0}^{t_{0}} e^{\textnormal{i}(t-\tau)\Delta} \Lambda_1(u)(\tau) \,d \tau\right\| _{L^{\infty}(\mathbb{R}^2)} & \lesssim \int_{0}^{t_{0}} \mid t-\tau\mid ^{-1} \|  \Lambda_1(u)(\tau)\| _{L^1(\mathbb{R}^2)} \,d \tau \\
& \lesssim \int_{0}^{t_{0}} \mid t-\tau\mid ^{-1} \int_0^\infty \frac{r}{1+r^2}  \mid u(\tau,r)\mid  \, dr d \tau.
\end{split}
\end{align}
Note that
$$
\int_0^\infty \frac{r}{1+r^2}  \mid u(\tau,r)\mid  \, dr \lesssim \| u(\tau)\| _{L^{p+1}(\mathbb{R}^2)}\left\|  (1+\mid x\mid ^2)^{-1}\right\| _{L^{(p+1)/p}(\mathbb{R}^2)} \lesssim \| u(\tau)\| _{L^{p+1}(\mathbb{R}^2)},
$$
where we used H\"older's inequality. Therefore, from \eqref{est.lambda1}, H\"older's inequality and Lemma \ref{lem.stest}, we deduce that
\begin{align}\label{est.1}
\begin{split}
\left\| \int_{0}^{t_{0}} e^{\textnormal{i}(t-\tau)\Delta} \Lambda_1(u)(\tau) \,d \tau\right\| _{L^{\infty}(\mathbb{R}^2)}
& \lesssim \left(\int_0^{t_0}\mid t-\tau\mid ^{-\frac{p+1}{p}} \, ds\right)^{\frac{p}{p+1}}\left( \int_{0}^{t_{0}} \| u(\tau)\| ^{p+1}_{L^{p+1}(\mathbb{R}^2)} \, d \tau\right)^{\frac{1}{p+1}}  \\
&\lesssim \mid t-t_0\mid ^{-\frac{1}{p+1}} \left( \int_{0}^{t_{0}}  \| u(\tau)\| ^{p+1}_{L^{p+1}(\mathbb{R}^2)} \, d \tau\right)^{\frac{1}{p+1}} \\
& \lesssim \delta^{-\frac{1}{p+1}} T^{-\frac{1-2\alpha}{p+1}}.
\end{split}
\end{align}
We next turn to treat the term with $\Lambda_2$. By using the definition of $\Lambda_2$ and Lemma \ref{lem.ineq}, we see that
$$
\mid \Lambda_2(t,r)\mid  \lesssim \frac{\mid u(t,r)\mid }{1+r^2}.
$$
Reasoning as the proof of \eqref{est.1}, we can similarly derive that
\begin{align}\label{est.2}
& \left\| \int_{0}^{t_{0}} e^{\textnormal{i}(t-\tau)\Delta} \Lambda_2(u)(\tau) \, d \tau\right\| _{L^{\infty}(\mathbb{R}^2)} \lesssim \delta^{-\frac{1}{p+1}} T^{-\frac{1-2\alpha}{p+1}}.
\end{align}
We now estimate the term with $\Lambda_3$. Utilizing standard dispersive estimate, H\"older's inequality and Lemma \ref{lem.stest}, we obtain that
\begin{align*}
\left\| \int_{0}^{t_{0}} e^{\textnormal{i} (t-\tau) \Delta} \Lambda_3(u)(\tau) \, d \tau \right\| _{L^{\infty}(\mathbb{R}^2)} =&\left\| \int_{0}^{t_{0}} e^{\textnormal{i} (t-\tau) \Delta}\mid u\mid ^{p-1} u \, d \tau\right\| _{L^{\infty}(\mathbb{R}^2)} \\
& \lesssim \int_{0}^{t_{0}}\mid t-\tau\mid ^{-1}\| u(\tau)\| _{L^{p+1}(\mathbb{R}^2)}^{\frac{(p+1)(p-2)}{p-1}}\| u(\tau)\| _{L^{2}(\mathbb{R}^2)}^{\frac{2}{p-1}} \, d \tau \\
& \lesssim \left(\int_0^{t_0} \mid t-\tau\mid ^{-(p-1)} \, d\tau \right)^{\frac{1}{p-1}} \left(\int_{0}^{t_{0}} \int_{\mathbb{R}^2}\mid u(\tau, x)\mid ^{p+1} \,d x d \tau\right)^{\frac{p-2}{p-1}}\\
& \lesssim \mid t-t_{0}\mid ^{-\frac{p-2}{p-1}} T^{\frac{\alpha(p-2)}{p-1}} \\
& \lesssim \delta^{-\frac{p-2}{p-1}}T^{-(1-2 \alpha) \frac{p-2}{p-1}}.
\end{align*}
This estimate together with \eqref{est.1}, \eqref{est.2} readily indicates that
\begin{align} \label{est.inter1}
\left\| \int_{0}^{t_{0}} e^{\textnormal{i}(t-\tau) \Delta} \Lambda(u)(\tau) \, d \tau\right\| _{L_{t, x}^{\infty}\left(\left[t_{1}, \infty\right) \times \mathbb{R}^{2}\right)} \lesssim \delta^{-\frac{1}{p+1}} T^{-\frac{1-2\alpha}{p+1}} + \delta^{-\frac{p-2}{p-1}}T^{-(1-2 \alpha) \frac{p-2}{p-1}}.
\end{align}
In addition, observe that
\[
\textnormal{i} \int_{0}^{t_{0}} e^{\textnormal{i}(t-\tau) \Delta} \Lambda(u)(\tau)\, d \tau=e^{\textnormal{i}\left(t-t_{0}\right) \Delta} u(t_{0})-e^{\textnormal{i} t \Delta} u_{0}.
\]
As a result of Strichartz estimates, we find that
\begin{align} \label{est.inter2}
\left\| \int_{0}^{t_{0}} e^{\textnormal{i}(t-\tau) \Delta}\Lambda(u)(\tau) \, d \tau\right\| _{L^q([0,\infty);L^r(\mathbb{R}^2))}
\lesssim  1.
\end{align}
Since $p>3$, by using \eqref{est.inter1}, \eqref{est.inter2} and interpolation inequality, then we get that there exists a constant $\beta>0$ such that for any $(1/q, 1/r) \in \mathfrak{T}_s$,
\begin{align} \label{est.s1}
\begin{aligned}
\left\| \int_{0}^{t_{0}} e^{\textnormal{i}(t-\tau) \Delta} \Lambda(u) (\tau) \,d \tau\right\| _{L^q([t_1,\infty);L^r(\mathbb{R}^2))}
\lesssim T^{-\beta}.
\end{aligned}
\end{align}

We next deal with the integral on $\left[t_{0}, t_{1}\right]$.
In this situation, we start with estimating the term with $\Lambda_3$.
In light of Sobolev embedding inequality with $0<s<1$ and Strichartz estimates for the couple $(1/q, 1/r) \in \mathfrak{T}_s$, we first derive that
\begin{align}\label{est.l333aa}
\begin{split}
&\left\| \int_{t_{0}}^{t_{1}} e^{\textnormal{i}(t-\tau) \Delta}\mid u\mid ^{p-1} u(\tau) \,d \tau\right\| _{L^q([t_1,\infty);L^r(\mathbb{R}^2))}   \lesssim\left\| \mid u\mid ^{p-1} u\right\| _{L_{(t_0,t_1)}^{q_1} W^{s,r_1}(\mathbb{R}^2)}
\end{split}
\end{align}
where
$$
\frac{1}{q_1} = \frac{1}{2}+\kappa,  \quad \frac{1}{r_1} = 1-\kappa
$$
and $\kappa >0$ is a small constant.
For any $0<\alpha <p-1$, by Lemma \ref{l.Nem2}, we find that
$$
\| \mid u\mid ^{p-1}u \| _{W^{s,r_1}(\mathbb{R}^2)} \lesssim \| u\| _{ W^{\sigma,\rho_1}(\mathbb{R}^2)}\| u\| ^{p-1}_{ L^{\rho_2(p-1)}(\mathbb{R}^2)},
$$
where $ s < \sigma < 1,$ $\rho_1 \geq 2$ and
$$
\frac{1}{r_1}=\frac{1}{\rho_1}+\frac{1}{\rho_2}.
$$
Using H\"older's inequality, then we obtain that
$$
\| \mid u\mid ^{p-1}u \| _{W^{s,r_1}(\mathbb{R}^2)}  \lesssim \| u\| _{ W^{\sigma,\rho_1}(\mathbb{R}^2)}\| u\| ^\alpha_{ L^{\alpha r_2}(\mathbb{R}^2)}\| u\| ^{p-1-\alpha}_{ L^{(p-1-\alpha) r_3}(\mathbb{R}^2)},
$$
where
$$
\frac{1}{r_2}+\frac{1}{r_3} =\frac{1}{\rho_2}.
$$
As a consequence, we have that
\begin{align} \label{l3}
\left\| \mid u\mid ^{p-1} u\right\| _{L_{(t_0,t_1)}^{q_1} W^{s,r_1}(\mathbb{R}^2)} \lesssim
\left\| u\right\| ^\alpha_{L_{(t_0,t_1)}^{\alpha q_1} L^{\alpha r_2}(\mathbb{R}^2)}
\left\| u\right\| ^{p-1-\alpha}_{L_{(t_0,t_1)}^{\infty} L^{\alpha r_3}(\mathbb{R}^2)} \left\| u\right\| _{L_{(t_0,t_1)}^{\infty} W^{\sigma,\rho_1}(\mathbb{R}^2)}.
\end{align}
Now we choose
$$
\alpha = (p+1) \left( \frac{1}{2} + \kappa\right), \quad \frac{1}{r_2} = \frac{1}{2} + \kappa, \quad \frac{1}{r_3} = \kappa, \quad \frac{1}{\rho_1}=\frac 12 -3\kappa.
$$
Then, by Sobolev embedding inequality, we conclude that
$$
\left\| u\right\| _{L_{(t_0,t_1)}^{\infty} L^{\alpha r_3}(\mathbb{R}^2)} \lesssim  \left\| u\right\| _{L_{(t_0,t_1)}^{\infty} H^1(\mathbb{R}^2)}, \quad  \left\| u\right\| _{L_{(t_0,t_1)}^{\infty} W^{\sigma,\rho_1}(\mathbb{R}^2)}\lesssim \left\| u\right\| _{L_{(t_0,t_1)}^{\infty} H^1(\mathbb{R}^2)}.
$$
Making use of \eqref{l3}, then we arrive at
$$
\left\| \mid u\mid ^{p-1} u\right\| _{L_{(t_0,t_1)}^{q_1} W^{s,r_1}(\mathbb{R}^2)} \lesssim
\left\| u\right\| ^\alpha_{L_{(t_0,t_1)}^{p+1} L^{p+1}(\mathbb{R}^2)} \left\| u\right\| ^{p-\alpha}_{L_{(t_0,t_1)}^{\infty} H^1(\mathbb{R}^2)}.
$$
From \eqref{est.small} and \eqref{est.l333aa}, then we deduce that
\begin{align}\label{est.l333}
\begin{split}
\hspace{-1cm}\left\| \int_{t_{0}}^{t_{1}} e^{\textnormal{i}(t-\tau) \Delta}\mid u\mid ^{p-1} u(\tau) \,d \tau\right\| _{L^2([t_1,\infty)); L^\infty(\mathbb{R}^2)}
\lesssim\| u\| ^{\alpha}_{L_{t, x}^{p+1}\left(\left[t_{0}, t_{1}\right] \times \mathbb{R}^{2}\right)}\left\|  u\right\| ^{p-\alpha}_{L_{t}^{\infty} H^1(\mathbb{R}^2)} =o(1).
\end{split}
\end{align}

We now turn to estimate the term with $\Lambda_2$.
From Sobolev embedding, Strichartz estimates and the definition of $\Lambda_2$, we deduce that
\begin{align} \label{2te1}
\begin{split}
&\left\| \int_{t_{0}}^{t_{1}} e^{\textnormal{i}(t-\tau) \Delta}\Lambda_2(u) \, d \tau\right\| _{L^{q}([t_1,\infty); L^r( \mathbb{R}^2))}\\
& \left\| \nabla \int_{t_{0}}^{t_{1}} e^{\textnormal{i}(t-s) \Delta} \Lambda_2(u)(s) \,d s\right\| _{L_{t}^{2(p-1)} L_{x}^{\frac{2(p-1)}{p-2}}}  +\left\|  \int_{t_{0}}^{t_{1}} e^{\textnormal{i}(t-s) \Delta} \Lambda_2(u)(s) \,d s\right\| _{L_{t}^{2(p-1)} L_{x}^{\frac{2(p-1)}{p-2}}}\\
& \lesssim \left\|  \int_{t_{0}}^{t_{1}} e^{\textnormal{i}(t-s) \Delta} \left(\nabla\left(\frac{A_\theta(\mid u\mid ^2)}{r} \right) \right)\left(\frac{A_\theta(\mid u\mid ^2)}{r}\right) u \,d s\right\| _{L_{t}^{2(p-1)} L_{x}^{\frac{2(p-1)}{p-2}}} \\
& \quad +\left\|  \int_{t_{0}}^{t_{1}} e^{\textnormal{i}(t-s) \Delta} \left(\frac{A_\theta(\mid u\mid ^2)}{r}\right)^2 \nabla u \,d s\right\| _{L_{t}^{2(p-1)} L_{x}^{\frac{2(p-1)}{p-2}}} \\
& \quad  + \left\|  \int_{t_{0}}^{t_{1}} e^{\textnormal{i}(t-s) \Delta}  \left(\frac{A_\theta(\mid u\mid ^2)}{r}\right)^2 u \,d s\right\| _{L_{t}^{2(p-1)} L_{x}^{\frac{2(p-1)}{p-2}}} \\
& \lesssim \left\| \nabla \left(\frac{A_\theta(\mid u\mid ^2)}{r} \right)\left( \frac{A_\theta(\mid u\mid ^2)}{r}\right) u \right\| _{L_{(t_0, t_1)}^{2 } L^1_x} + \left\|  \left(\frac{A_\theta(\mid u\mid ^2)}{r}\right)^2 \nabla u \right\| _{L_{(t_0, t_1)}^{2} L^1_x} \\
& \quad + \left\|   \left(\frac{A_\theta(\mid u\mid ^2)}{r}\right)^2 u \right\| _{L_{(t_0, t_1)}^{2} L^1_x}.
\end{split}
\end{align}
We are now going to estimate every term in the right side hand of \eqref{2te1}. Let us begin with handling the first term. Since
$$
\nabla\left(\frac{A_\theta(\mid u\mid ^2)}{r} \right)= - \frac{A_\theta(\mid u\mid ^2)}{r^2}- \frac{\mid u\mid ^2}{2},
$$
by applying Lemma \ref{lem.ineq} and the uniform boundedness of $H^1$-norm of the solution, then we deduce that
$$
\left\| \nabla \left(\frac{A_\theta(\mid u\mid ^2)}{r} \right)\right\| _{L_{(t_0,t_1)}^{\infty } L_{x}^{2}} \lesssim 1.
$$
Using Corollary \ref{cor.sm}, then we have that
\begin{align}\label{eq.sm2}
\begin{split}
\hspace{-2cm}\left\| \nabla\left(\frac{A_\theta(\mid u\mid ^2)}{r}\right)\left( \frac{A_\theta(\mid u\mid ^2)}{r}\right) u \right\| _{L_{(t_0,t_1)}^{2 } L_{x}^{1}}  &\lesssim \left\|  \nabla \left(\frac{A_\theta(\mid u\mid ^2)}{r} \right)\right\| _{L_{(t_0,t_1)}^{\infty } L_{x}^{2}}
\left\|   \left( \frac{A_\theta(\mid u\mid ^2)}{r}\right) u \right\| _{L_{(t_0,t_1)}^{2 } L_{x}^{2}} \\
& \lesssim \varepsilon.
\end{split}
\end{align}
Next we treat the second term in the right side hand of \eqref{2te1}. To do this,
we split the integration domain $\mathbb{R}^2 $ into two subdomains $\{\mid x\mid  \leq T^\alpha\}$ and $\{\mid x\mid  > T^\alpha\}$. It yields from H\"older's inequality that
$$
\left\| \left(\frac{A_\theta(\mid u\mid ^2)}{r}\right)^2 \nabla u \right\| _{L_{(t_0,t_1)}^{2} L_{\{\mid x\mid  \leq T^\alpha\}}^{1}} \leq
\left\| \left(\frac{A_\theta(\mid u\mid ^2)}{r}\right)^2  \right\| _{L_{(t_0,t_1)}^{\infty} L_{x}^{2}}
\left\| \nabla u \right\| _{L_{(t_0,t_1)}^{2} L_{\{\mid x\mid  \leq T^\alpha\}}^{2}}.
$$
In view of Corollary \ref{cor.sm}, we know that
$$
\left\| \nabla u \right\| _{L_{(t_0,t_1)}^{2} L_{\{\mid x\mid  \leq T^\alpha\}}^{2}} \lesssim \varepsilon.
$$
On the other hand, by using Lemma \ref{lem.ineq} and the uniform boundedness of $H^1$-norm of the solution, we obtain that
$$
\left\| \left(\frac{A_\theta(\mid u\mid ^2)}{r}\right)^2  \right\| _{L_{(t_0,t_1)}^{\infty} L_{x}^{2}}
\lesssim 1.
$$
Hence we deduce that
\begin{align}\label{eq.sm4}
\left\| \left(\frac{A_\theta(\mid u\mid ^2)}{r}\right)^2 \nabla u \right\| _{L_{(t_0,t_1)}^{2} L_{\{\mid x\mid \leq T^\alpha\}}^{1}} \lesssim \varepsilon.
\end{align}
We now estimate the integral on the exterior domain $\{\mid x\mid  > T^\alpha\}$.  Note first that $\left\| {A_\theta(\mid u\mid ^2)}\right\| _{L^{\infty}(\mathbb{R}^2)}\lesssim 1$, see Lemma \ref{lem.ineq}. Since $1/3 < \sigma<1/2$, then it is not hard to verify that
\begin{align*}
\left\| \left(\frac{A_\theta(\mid u\mid ^2)}{r}\right)^2 \nabla u \right\| _{L_{(t_0,t_1)}^{2} L_{\{\mid x\mid > T^\alpha\}}^{1}}
&\leq\left\| \left(\frac{A_\theta(\mid u\mid ^2)}{r}\right)^2  \right\| _{L_{(t_0,t_1)}^{2} L_{\{\mid x\mid > T^\alpha\}}^{2}}
\left\| \nabla u \right\| _{L_{(t_0,t_1)}^{\infty} L_{x}^{2}} \\
&\lesssim \left\| \frac{1}{r^2}\right\| _{L_{(t_0,t_1)}^{2} L_{\{\mid x\mid > T^\alpha\}}^{2}}
\lesssim \delta^{1/2} T^{(1-3\alpha)/2}=o(1).
\end{align*}
This along with \eqref{eq.sm4} yields that
\begin{align} \label{est.l212}
\left\| \left(\frac{A_\theta(\mid u\mid ^2)}{r}\right)^2 \nabla u \right\| _{L_{(t_0,t_1)}^{2} L_{x}^{1}}=o(1).
\end{align}
We finally estimate the third term in the right side hand of \eqref{2te1}.
In light of Corollary \ref{cor.sm}, we get that
\begin{align*}
\left\| \left(\frac{A_\theta(\mid u\mid ^2)}{r}\right)^2 u \right\| _{L_{(t_0,t_1)}^{2} L^{1}_{\{\mid x\mid \leq T^\alpha\}}}
&\lesssim \left\| \left(\frac{A_\theta(\mid u\mid ^2)}{r}\right) u \right\| _{L_{(t_0,t_1)}^{2} L_{x}^{2}}\left\|    \frac{A_\theta(\mid u\mid ^2)}{r} \right\| _{L_{(t_0,t_1)}^{\infty} L^{2}_{\{\mid x\mid \leq T^\alpha\}}} \\
& \lesssim \varepsilon \left\|    \frac{1}{r} \right\| _{L_{(t_0,t_1)}^{\infty} L^{2}_{\{\mid x\mid  \leq T^\alpha\}}} = \varepsilon  \ln^{1/2} T.
\end{align*}
Using Remark \ref{rem.des}, then we obtain that
\begin{align}\label{eq.sm0}
 \left\| \left(\frac{A_\theta(\mid u\mid ^2)}{r}\right)^2 u \right\| _{L_{(t_0,t_1)}^{2} L^{1}_{\{ \mid x\mid  \leq T^\alpha\}}} = o(1).
\end{align}
Furthermore, we can similarly deduce that
\begin{align*}
\left\| \left(\frac{A_\theta(\mid u\mid ^2)}{r}\right)^2  u \right\| _{L_{(t_0,t_1)}^{2} L_{\{\mid x\mid  > T^\alpha\}}^{1}}
&\leq \left\| \left(\frac{A_\theta(\mid u\mid ^2)}{r}\right)^2  \right\| _{L_{(t_0,t_1)}^{2} L_{\{\mid x\mid  > T^\alpha\}}^{2}}
\left\|  u \right\| _{L_{(t_0,t_1)}^{\infty} L_{x}^{2}} \\
&\lesssim \left\| \frac{1}{r^2}\right\| _{L_{(t_0,t_1)}^{2} L_{\{\mid x\mid > T^\alpha\}}^{2}}
\lesssim \delta^{1/2} T ^{(1-3\alpha)/2}=o(1).
\end{align*}
This together with \eqref{eq.sm0} leads to
\begin{align} \label{est.l23}
\left\| \left(\frac{A_\theta(\mid u\mid ^2)}{r}\right)^2 u \right\| _{L_{(t_0,t_1)}^{2} L^{1}_x}=o(1).
\end{align}
Combining \eqref{2te1}, \eqref{eq.sm2}, \eqref{est.l212} and \eqref{est.l23}, then we deduce
\begin{align} \label{est.l222}
\left\| \int_{t_0}^{t_{1}} e^{\textnormal{i}(t-\tau) \Delta} \Lambda_2(u) (\tau) \,d \tau\right\| _{L^q([t_1,\infty);L^r(\mathbb{R}^2))}
\lesssim o(1).
\end{align}

Let us now estimate the term with $\Lambda_1$. From Sobolev embedding, Strichartz estimates and the definition of $\Lambda_1$, we deduce that
 \begin{align} \label{est.l11} \nonumber
&\left\| \int_{t_0}^{t_{1}} e^{\textnormal{i}(t-\tau) \Delta} \Lambda_1(u) (\tau) \,d \tau\right\| _{L^q([t_1,\infty); L^r(\mathbb{R}^2))} \\ \nonumber
&\lesssim  \left\| \nabla \int_{t_{0}}^{t_{1}} e^{\textnormal{i}(t-\tau) \Delta} \Lambda_1(u)(\tau) \,d \tau\right\| _{L_{t}^{2(p-1)} L_{x}^{\frac{2(p-1)}{p-2}}} +\left\| \int_{t_{0}}^{t_{1}} e^{\textnormal{i}(t-\tau) \Delta} \Lambda_1(u)(\tau) \,d \tau\right\| _{L_{t}^{2(p-1)} L_{x}^{\frac{2(p-1)}{p-2}}} \\\nonumber
&\lesssim  \left\| \int_{t_{0}}^{t_{1}} e^{\textnormal{i}(t-\tau) \Delta} \left(\nabla A_0(\mid u\mid ^2) \right) u \,d \tau\right\| _{L_{t}^{2(p-1)} L_{x}^{\frac{2(p-1)}{p-2}}} +\left\|  \int_{t_{0}}^{t_{1}} e^{\textnormal{i}(t-\tau) \Delta} A_0(\mid u\mid ^2) \nabla u \,d \tau \right\| _{L_{t}^{2(p-1)} L_{x}^{\frac{2(p-1)}{p-2}}} \\\nonumber
&\quad +\left\|  \int_{t_{0}}^{t_{1}} e^{\textnormal{i}(t-\tau) \Delta} A_0(\mid u\mid ^2) u \,d \tau \right\| _{L_{t}^{2(p-1)} L_{x}^{\frac{2(p-1)}{p-2}}} \\
&\lesssim \left\|   \left(\nabla A_0(\mid u\mid ^2) \right) u \right\| _{L_{(t_0,t_1)}^{2 } L_{x}^{1}} + \left\|    A_0(\mid u\mid ^2) \nabla u \right\| _{L_{(t_0,t_1)}^{2} L_{x}^{1}}  + \left\|    A_0(\mid u\mid ^2) u \right\| _{L_{(t_0,t_1)}^{2} L_{x}^{1}}.
\end{align}
We are going to estimate every term in the right side hand of  \eqref{est.l11}. Let us start with handling the first term in the right side hand of \eqref{est.l11}.
Notice that
$$
\nabla A_0(\mid u\mid ^2)= \frac{A_\theta(\mid u\mid ^2)}{r} \mid u\mid ^2,
$$
then
$$
\left\|  \left(\nabla A_0(\mid u\mid ^2) \right) u \right\| _{L_{(t_0,t_1)}^{2 } L_{x}^{1}} = \left\|  \frac{A_\theta(\mid u\mid ^2)}{r} \mid u\mid ^3\right\| _{L_{(t_0,t_1)}^{2 } L_{x}^{1}} \leq \left\|  \frac{A_\theta(\mid u\mid ^2)}{r} u\right\| _{L_{(t_0,t_1)}^{2 } L_{x}^{2}} \left\|  u\right\| ^2_{L_{(t_0,t_1)}^{\infty } L_{x}^{4}}.
$$
From Corollary \ref{cor.sm} and the uniform boundedness of $H^1$-norm of the solution, then we conclude that
\begin{align}\label{eq.a0f1}
\left\| \left(\nabla A_0(\mid u\mid ^2) \right) u \right\| _{L_{(t_0,t_1)}^{2 } L_{x}^{1}} = o(1).
\end{align}
Next we deal with the second term in the right side hand of \eqref{est.l11}.
We split again the integration domain on $\mathbb{R}^2$ into two subdomains $\{\mid x\mid  \leq T^\alpha\}$ and $\{\mid x\mid >T^\alpha\}$.
Observe that
$$
\left\|    A_0(\mid u\mid ^2) \nabla u \right\| _{L_{(t_0,t_1)}^{2} L_{\{\mid x\mid  \leq T^\alpha\}}^{1}} \leq  \left\|   A_0(\mid u\mid ^2) \right\| _{L_{(t_0,t_1)}^{\infty} L_{\{\mid x\mid  \leq T^\alpha\}}^{2}}  \left\|   \nabla u \right\| _{L_{(t_0,t_1)}^{2} L_{x}^{2}}.
$$
According to Corollary \ref{cor.sm}, we know that
$$
\left\|    \nabla  u \right\| _{L_{(t_0,t_1)}^{2} L_{\{\mid x\mid  \leq T^\alpha\}}^{2}} \leq \varepsilon.
$$
Moreover, by using Lemma \ref{lem.bf1}, we have that
$$
\left\| A_0(\mid u\mid ^2) \right\| _{L_{(t_0,t_1)}^{\infty} L_{ \{\mid x\mid  \leq T^\alpha\}}^{2}} \lesssim 1.
$$
Therefore, we find that
$$
\left\|  A_0(\mid u\mid ^2) \nabla u \right\| _{L_{(t_0,t_1)}^{2} L_{\{ \mid x\mid  \leq T^\alpha\}}^{1}} \leq \varepsilon.
$$
On the other hand, taking into account Lemma \ref{lem.bf1} and the fact $1/3 < \sigma<1/2$, we are able to derive that
$$
\left\|  A_0(\mid u\mid ^2) \nabla u \right\| _{L_{(t_0,t_1)}^{2} L_{\{\mid x\mid >T^\alpha\}}^{1}} \lesssim \left\|  \frac{1}{r^2} \right\| _{L_{(t_0,t_1)}^{2} L_{\{\mid x\mid >T^\alpha\}}^{2}} \left\| \nabla u \right\| _{L_{(t_0,t_1)}^{\infty} L_{x}^{2}}=o(1).
$$
As a consequence, we get that
\begin{align}\label{eq.a0s2}
\left\|    A_0(\mid u\mid ^2) \nabla  u \right\| _{L_{(t_0,t_1)}^{2} L_{x}^{1}} = o(1).
\end{align}
We now treat the third term in the right side hand of \eqref{est.l11}. In the interior domain $\{\mid x\mid  \leq T^\alpha\}$, we have that
 $$
 \left\|    A_0(\mid u\mid ^2) u \right\| _{L_{(t_0,t_1)}^{2} L_{\{\mid x\mid <T^\alpha\}}^{1}} \leq  \left\|    \sqrt{A_0(\mid u\mid ^2)} \right\| _{L_{(t_0,t_1)}^{\infty} L_{\{\mid x\mid < T^\alpha\}}^{2}}  \left\| \sqrt{A_0(\mid u\mid ^2)}\ u \right\| _{L_{t}^{2} L_{x}^{2}}.
 $$
From Corollary \ref{cor.sm},  we obtain that
$$
\left\|  \sqrt{A_0(\mid u\mid ^2)}\,u \right\| _{L_{(t_0,t_1)}^{2} L_{x}^{2}} \leq \varepsilon.
$$
Note that
$$
\left\|  \sqrt{A_0(\mid u\mid ^2)} \right\| _{L_{(t_0,t_1)}^{\infty} L_{\{\mid x\mid < T^\alpha\}}^{2}} \leq \left\|  (1+r)^{-2} \right\| _{L_{t}^{\infty} L_{\{\mid x\mid  < T^\alpha\}}^{1}}^{\frac 12} \lesssim \ln^{1/2} T.
$$
Hence we see that
$$
\left\| A_0(\mid u\mid ^2) u \right\| _{L_{(t_0,t_1)}^{2} L_{\{\mid x\mid < T^\alpha\}}^{1}}  \lesssim  \varepsilon \ln^{1/2} T.
$$
Now we make use of Remark \ref{rem.des} to deduce that
\begin{align}\label{eq.smot1}
\left\|  A_0(\mid u\mid ^2) u \right\| _{L_{(t_0,t_1)}^{2} L_{\{\mid x\mid < T^\alpha\}}^{1}} = o(1).
\end{align}
For the exterior domain  $\{\mid x\mid  >T^\alpha\}$, by applying Lemma \ref{lem.bf1}, we conclude that
$$
\left\|    A_0(\mid u\mid ^2) u \right\| _{L_{(t_0,t_1)}^{2} L_{\{\mid x\mid >T^\alpha\}}^{1}} \leq \left\|  \frac{1}{r^2} \right\| _{L_{(t_0,t_1)}^{2} L_{\{\mid x\mid >T^\alpha\}}^{2}} \left\| u \right\| _{L_{(t_0,t_1)}^{\infty} L_{x}^{2}}=o(1).
$$
This jointly with \eqref{eq.smot1} shows that
\begin{align} \label{est.l1111}
\left\|  A_0(\mid u\mid ^2) u \right\| _{L_{(t_0,t_1)}^{2} L_x^1}=o(1).
\end{align}
Combining \eqref{est.l11}, \eqref{eq.a0f1}, \eqref{eq.a0s2} and \eqref{est.l1111}, then we get that
\begin{align} \label{est.l111}
\begin{aligned}
&{
\left\| \int_{t_0}^{t_{1}} e^{\textnormal{i}(t-\tau) \Delta} \Lambda_1(u) (\tau) \,d \tau\right\| _{L^q((t_1,\infty);L^r(\mathbb{R}^2))}
}
=o(1).
\end{aligned}
\end{align}
From \eqref{est.l333}, \eqref{est.l222} and \eqref{est.l111}, then we arrive at
\begin{align*}
{
\left\| \int_{t_0}^{t_{1}} e^{\textnormal{i}(t-\tau) \Delta} \Lambda(u) (\tau) \,d \tau\right\| _{L^q((t_1,\infty);L^r(\mathbb{R}^2))}
}
 \lesssim  o(1).
\end{align*}
This together with \eqref{est.s1} readily suggests that \eqref{eq.4.4} holds true. Thus we have completed the proof of Theorem \ref{thm.scattering1} $\textnormal{i)}$.
\end{proof}

\subsection{Blow-up of solutions} We now discuss finite time blow-up of solutions to the Cauchy problem for \eqref{CSS} with initial data in $\mathcal{K}^-$. For this, we first give equivalent variational characterization of the ground state energy $d$.

\begin{lem} \label{lem.ide}
Let $p>3$, then
\begin{align*}
\begin{split}
d&= \inf \left\{L(u): u \in H^1_{rad}(\mathbb{R}^2) \backslash \{0\}, \, K(u) \leq 0\right\} \\
&=  \inf \left\{L(u): u \in H^1_{rad}(\mathbb{R}^2) \backslash \{0\}, \, K(u)  < 0\right\},
\end{split}
\end{align*}
where
$$
L(u)=S(u)-\frac 12 K(u).
$$
\end{lem}
\begin{proof}
We start with deducing that the first identity holds true. If $K(u)=0$, then $L(u)=S(u)$. In view of the definition of $d$, then we know that
\begin{align}\label{ide.d1}
d=\inf \left\{L(u): u \in H^1_{rad}(\mathbb{R}^2) \backslash \{0\}, \, K(u)  = 0\right\}.
\end{align}
This implies that
$$
d \geq  \inf \left\{L(u): u \in H^1_{rad}(\mathbb{R}^2) \backslash \{0\}, \, K(u)  \leq 0\right\}.
$$
On the other hand, if $ K(u) \leq 0$, then it is not hard to conclude from \eqref{ide.k} that there exists $\lambda \in (0, 1]$ such that $K(u_{\lambda})=0$. Thanks to $p>3$, then
\begin{align*}
d \leq S(u_{\lambda}) & =S(u_{\lambda}) - \frac 12 K(u_{\lambda}) \\
&= \frac 12\| u\| _{L^{2}(\mathbb{R}^2)}^{2} + \frac{\lambda^{p-1}(p-3)}{2(p+1)}  \| u\| _{L^{p+1}(\mathbb{R}^2)}^{p+1} \\
& \leq \frac 12\| u\| _{L^{2}(\mathbb{R}^2)}^{2} + \frac{p-3}{2(p+1)}  \| u\| _{L^{p+1}(\mathbb{R}^2)}^{p+1} \\
& = S(u) - \frac 12 K(u)  \\
&=L(u).
\end{align*}
This gives that
$$
d \leq  \inf \left\{L(u): u \in H^1_{rad}(\mathbb{R}^2) \backslash \{0\}, \, K(u)  \leq 0\right\}.
$$
Consequently, we have that
\begin{align} \label{ide.d}
d =\inf \left\{L(u): u \in H^1_{rad}(\mathbb{R}^2) \backslash \{0\}, \, K(u)  \leq 0\right\}.
\end{align}
We next prove that the second identity holds true. If $K(u)=0$, then we can derive from \eqref{ide.k} that there exists $\tau >1$ such that $K(u_{\tau}) <0$. In addition, there holds that $u_{\tau} \to u$ in $H^1(\mathbb{R}^2)$ as $\tau \to 1^+$. Accordingly, we have that
\begin{align*}
\inf \left\{L(u): u \in H^1_{rad}(\mathbb{R}^2) \backslash \{0\}, \, K(u)  < 0\right\} \leq L(u_{\tau})  \to L(u) \quad \mbox{as} \,\,\, \tau \to 1^+.
\end{align*}
This jointly with \eqref{ide.d1} infers that
\begin{align} \label{ineq.d11}
\inf \left\{L(u): u \in H^1_{rad}(\mathbb{R}^2) \backslash \{0\}, \, K(u)  <0\right\} \leq d.
\end{align}
Observe that
$$
\inf \left\{L(u): u \in H^1_{rad}(\mathbb{R}^2) \backslash \{0\}, \, K(u)  \leq 0\right\} \leq \inf \left\{L(u): u \in H^1_{rad}(\mathbb{R}^2) \backslash \{0\}, \, K(u)  < 0\right\}.
$$
Taking into account \eqref{ide.d} and \eqref{ineq.d11}, then we obtain that
$$
\inf \left\{L(u): u \in H^1_{rad}(\mathbb{R}^2) \backslash \{0\}, \, K(u)   \leq 0\right\} = \inf \left\{L(u): u \in H^1_{rad}(\mathbb{R}^2) \backslash \{0\}, \, K(u)  <  0\right\},
$$
and the proof is completed.
\end{proof}

In order to investigate finite time blow-up of solutions, we now need to introduce a localized virial quantity, which is inspired by the ideas developed in \cite{BL}. Let $\chi : \mathbb{R}^2 \to \R$ be a smooth function with regularity property $\nabla^k \chi \in L^{\infty}(\mathbb{R}^2)$ for $1 \leq k \leq 4$ such that
\begin{equation*}
\chi(r)=\left\{
\begin{array}{lr}
\frac{r^2}{2}  \quad  \quad \quad \,\, \, \text{for} \, \, r \leq 1,\\
\text{const.}  \quad \, \,\,\, \, \,  \text{for} \, \, r \geq 10,
\end{array}
\quad \mbox{and} \,\,\, \chi''(r) \leq 1 \,\, \mbox{for any} \,\,r \geq 0.
\right.
\end{equation*}
For $R>0$, we define a function $\chi_R: \R \to \R$ by
\begin{align*}
\chi_R(r)=R^2 \chi\left(\frac{r}{R}\right).
\end{align*}
It is simple to check that
\begin{align*}
1- \chi_R''(r) \geq 0, \,\,\, 1 -\frac{\chi_R'(r)}{r} \geq 0, \,\,\, 2- \Delta \chi_R(r) \geq 0 \,\, \, \mbox{for any} \,\, \, r \geq 0.
\end{align*}
We now introduce the associated localized virial quantity as
\begin{align} \label{def.virial}
V_{\chi_R}[u(t)]=\mbox{Im} \int_{\mathbb{R}^2} \overline{u} \left(D_1 u \, \partial_1 \chi_R + D_2 u \, \partial_2 \chi_R \right) \, dx.
\end{align}

\begin{lem} \label{lem.virial}
Let $u \in C([0, T_{\textnormal{max}}), H^1_{rad}(\mathbb{R}^2))$ be a solution to the Cauchy problem for \eqref{CSS}, then, for any $t \in [0, T_{\textnormal{max}})$,
\begin{align*}
\frac{d}{dt}V_{\chi_R}[u(t)] \lesssim 2K(u(t)) +  R^{-\frac{p-1}{2}} \left(\| D_1 u(t)\| ^{\frac{p-1}{2}}_{L^2(\mathbb{R}^2)} + \| D_2 u(t)\| ^{\frac{p-1}{2}}_{L^2(\mathbb{R}^2)}\right) +  R^{-2}.
\end{align*}
\end{lem}
\begin{proof}
This result is a direct consequence of \cite[Lemma 6.4]{Gou}, then we omit its proof here.
\end{proof}

With Lemmas \ref{lem.ide} and \ref{lem.virial} in hand, we are now ready to prove finite time blow-up of solutions to the Cauchy problem for \eqref{CSS} with initial data in $\mathcal{K}^+$.

\begin{proof}[Proof of Theorem \ref{thm.scattering1} $\textnormal{ii)}$] For simplicity, we shall write $u=u(t)$ in the following. By a similar way as the proof of Theorem \ref{thm.scattering1} $\textnormal{i)}$, we can easily show that the set $\mathcal{K}^-$ is invariant under the flow of the Cauchy problem for \eqref{CSS}. Since $u_0 \in \mathcal{K}^-$, then $K(u)<0$ for any $t \in [0, T_{\textnormal{max}})$. Thus, from \eqref{ide.k}, we know that there is a constant $0<\lambda^* <1$ such that $K(u_{\lambda^*})=0$. Moreover, it is obvious that the function $\lambda \mapsto S(u_{\lambda})$ is concave on $[\lambda^*, 1]$. Hence
\begin{align*}
S(u_{\lambda^*}) -S(u) \leq \left( \lambda^*-1\right) \frac{d}{d \lambda} S(u_{\lambda}){\mid_{\lambda=1}} = \left( \lambda^*-1\right) K(u).
\end{align*}
Noting that $K(u) <0$ and $d \leq S(u_{\lambda^*})$, then we obtain that
\begin{align*}
K(u) \leq \left(1-\lambda^*\right) K(u) \leq S(u)-S(u_{\lambda^*})  \leq S(u)- d.
\end{align*}
This indicates that
\begin{align} \label{lbdd}
\| D_1 u\| ^{2}_{L^2{(\mathbb{R}^2)}} + \| D_2 u \| ^{2}_{L^2{(\mathbb{R}^2)}} \leq \frac{p-1}{p+1} \| u\| ^{p+1}_{L^{p+1}{(\mathbb{R}^2)}} + S(u) -d.
\end{align}
Since $p \leq 5$, then, by using Young's inequality, we can deduce from Lemma \ref{lem.virial} that, for any $\eps>0$, there exists $R>0$ sufficiently large such that
\begin{align*}
\frac{d}{dt}V_{\chi_R}[u] & \lesssim 2K(u) + \eps \left(\| D_1 u\| ^{2}_{L^2(\mathbb{R}^2)} + \| D_2 u\| ^{2}_{L^2(\mathbb{R}^2)}\right) + \eps,
\end{align*}
which suggests that
\begin{align*}
\frac{d}{dt}V_{\chi_R}[u]  & \lesssim 4 S(u)-2 \| u\| _{L^2(\mathbb{R}^2)}^2 + \frac{6-2p}{p+1} \| u\| _{{L^{p+1}(\mathbb{R}^2)}}^{p+1}  \\
& \quad + \eps \left(\| D_1 u\| ^{2}_{L^2(\mathbb{R}^2)} + \| D_2 u\| ^{2}_{L^2(\mathbb{R}^2)}\right) + \eps.
\end{align*}
Applying \eqref{lbdd}, we now get that
\begin{align} \label{vinequ}
\begin{split}
\frac{d}{dt}V_{\chi_R}[u] & \lesssim  4 S(u)-2 \| u\| _{L^2(\mathbb{R}^2)}^2 + \frac{6-2p+ \eps(p-1)}{p+1} \| u\| _{{L^{p+1}(\mathbb{R}^2)}}^{p+1} \\
& \quad + \eps \left(S(u)-d\right) + \eps.
\end{split}
\end{align}
Due to $K(u)<0$, by Lemma \ref{lem.ide}, we know that
$$
d \leq L(u)=\frac{p-3}{2(p+1)} \| u\| _{{L^{p+1}(\mathbb{R}^2)}}^{p+1} + \frac 12 \| u\| _{L^2(\mathbb{R}^2)}^2,
$$
from which we infer that
\begin{align}\label{ineq.p}
2d-\| u\| _{L^2(\mathbb{R}^2)}^2 \leq  \frac{p-3}{p+1} \| u\| _{L^{p+1}(\mathbb{R}^2)}^{p+1}.
\end{align}
Since $p>3$, then, for any $\eps>0$ sufficiently small,
$$
\frac{6-2p+ \eps(p-1)}{p-3} <0.
$$
Making use of \eqref{ineq.p}, then we obtain from \eqref{vinequ} that
\begin{align*}
\frac{d}{dt}V_{\chi_R}[u]  \leq  4 S(u)-2 \| u\| _{L^2(\mathbb{R}^2)}^2 + \frac{6-2p+ \eps(p-1)}{p-3}  \left(2d-  \| u\| _{L^2(\mathbb{R}^2)}^2\right) +  \eps \left(S(u)-d\right) + \eps.
\end{align*}
Noting that $S(u) <d$, by the conservation laws, we derive that there exists a constant $\delta>0$ small such that $S(u) \leq (1 - \delta) d$. Therefore, we get that
\begin{align*}
\frac{d}{dt}V_{\chi_R}[u] &\lesssim - \left(4 \delta + \eps \delta- \frac{2\eps(p-1)}{p-3} \right)d - \frac{\eps(p+1)}{p-3}  \| u\| _{L^2(\mathbb{R}^2)}^2 + \eps \\
&\lesssim -\left(4 \delta + \eps \delta-\frac{2\eps(p-1)}{p-3}  \right)d + \eps.
\end{align*}
For any $\eps>0$ sufficiently small, then we have that
$$
\frac{d}{dt}V_{\chi_R}[u] \lesssim -2\delta  d.
$$
This readily indicates that $u$ has to blow up in finite time, and the proof is completed.
\end{proof}

\section{Appendix}

\subsection{Smallness of solutions} We shall prove that \eqref{eq.key1m} holds by assuming that
\begin{align}\label{eq.sm21m}
\left\| e^{\textnormal{i} \left(t-t_{1}\right) \Delta} u\left(t_{1}\right)\right\| _{L^{q}\left([t_1,\infty); L^r(\mathbb{R}^{2})\right)}   \lesssim o(1),
\end{align}
where $t_1 = t_1(\varepsilon) \to \infty$ as $\varepsilon \to 0$ and $(1/q, 1/r) \in \mathfrak{T}_s$ with $s \in (0,1)$ close to $1$. For this, we need to establish the following result.







\begin{prop}\label{p.mca1}
Let $(1/q, 1/r) \in \mathfrak{T}_s$ with $s \in (0,1)$ close to $1$, then
\begin{align*}
\begin{aligned}
\left\|  \int_{t_1}^t e^{-\mathrm{i} (t-\tau)\Delta} \Lambda(u(\tau)) \,ds\right\| _{L^{q}\left([t_1,\infty); L^r(\mathbb{R}^{2})\right)}  &\lesssim  \| u\| _{L_{[t_1, \infty)}^4 L_x^4}^2 + \| u\| _{L_{[t_1, \infty)}^6 L_x^6}^3 + \| u\| _{L_{[t_1, \infty)}^8 L_x^{8/3}}^4 \\
 & \quad +\left\|  u \right\| _{L^{2(p-1) }([t_1, \infty)\times \mathbb{R}^2)}^{p-1}.
\end{aligned}
\end{align*}
\end{prop}
\begin{proof}
Note first that for any $(1/q,1/r) \in \mathfrak{T}_s$ there exists a Strichartz couple
$(1/q_1, 1/r_1) $ with $1/q_1+1/r_1 = 1/2$ and $(q_1, r_1) \neq (2,\infty)$ such that
\begin{align*}
&\left\|  \int_{t_1}^t e^{\mathrm{i} (t-\tau)\Delta} \Lambda(u(\tau)) \,ds\right\| _{L^{q}\left([t_1,\infty); L^r(\mathbb{R}^{2})\right)} \\
&\lesssim \left\|  \nabla \int_{t_1}^t e^{\mathrm{i} (t-\tau)\Delta} \Lambda(u(\tau)) \,ds\right\| _{L^{q_1}\left([t_1,\infty); L^{r_1}(\mathbb{R}^{2})\right)}+ \left\|   \int_{t_1}^t e^{\mathrm{i} (t-\tau)\Delta} \Lambda(u(\tau)) \,ds\right\| _{L^{q_1}\left([t_1,\infty); L^{r_1}(\mathbb{R}^{2})\right)}.
\end{align*}
The application of Strichartz estimates further leads to
\begin{align}\label{2pm21}
\begin{split}
\hspace{-1cm}\left\| \int_{t_{1}}^{t} e^{\mathrm{i}(t-\tau) \Delta}\Lambda(u)\,d \tau\right\| _{L^{q}\left([t_1,\infty); L^r(\mathbb{R}^{2})\right)}
&\lesssim \left\| \nabla \Lambda(u) \right\| _{L^{2 }([t_1,\infty); L^1(\mathbb{R}^2))} + \left\| \Lambda(u) \right\| _{L^{2 }([t_1,\infty); L^1(\mathbb{R}^2))} \\
&=\left\|  \Lambda_1(u) \right\| _{L^{2 }([t_1,\infty); W^{1,1}(\mathbb{R}^2))} + \left\| \Lambda_2(u) \right\| _{L^{2 }([t_1,\infty); W^{1,1}(\mathbb{R}^2))} \\
& \quad + \left\|  \Lambda_3(u) \right\| _{L^{2 }([t_1,\infty); W^{1,1}(\mathbb{R}^2))}.
\end{split}
\end{align}

In the following, we shall evaluate every term in the right side hand of \eqref{2pm21}. Let us first hand the local term $\Lambda_3$. It is simple to deduce that
$$
\left\| \Lambda_3(u) \right\| _{L^{2 }([t_1,\infty); W^{1,1}(\mathbb{R}^2))} \lesssim \| u \| _{L^{\infty }([t_1,\infty); H^1(\mathbb{R}^2))}\left\| u \right\| _{L^{2(p-1) }([t_1,\infty) \times \mathbb{R}^2)}^{p-1} \lesssim \left\| u \right\| _{L^{2(p-1) }([t_1,\infty) \times \mathbb{R}^2)}^{p-1}.
$$
We next treat the nonlocal term $\Lambda_2$. In virtue of Lemma \ref{lem.ineq}, we first have that
\begin{align}\label{eq.esmk1}
 \left\| \frac{A_\theta(\mid u\mid ^2)}{r^2}\right\| _{L^q(\mathbb{R}^2)} \leq \| u\| ^2_{L^{2q}(\mathbb{R}^2)}, \quad 1 < q < \infty
\end{align}
and
\begin{align}\label{eq.esmk2}
\left\| A_\theta(\mid u\mid ^2)\right\| _{L^\infty(\mathbb{R}^2)} \lesssim \| u\| ^2_{L^{2} (\mathbb{R}^2)}.
\end{align}
Observe that
$$
\nabla \Lambda_2(u)=2u\left(\frac{A_\theta(\mid u\mid ^2)}{r}\right) \nabla\left(\frac{A_\theta(\mid u\mid ^2)}{r}\right)+\left(\frac{A_\theta(\mid u\mid ^2)}{r}\right)^2 \nabla u
$$
and
$$
\nabla\left(\frac{A_\theta(\mid u\mid ^2)}{r} \right)= - \frac{A_\theta(\mid u\mid ^2)}{r^2}- \frac{\mid u\mid ^2}{2}.
$$
Therefore, we get that
\begin{align*}
\left\| \Lambda_2(u) \right\| _{L^{2 }([t_1,\infty); W^{1,1}(\mathbb{R}^2))} &\lesssim \left\| \nabla\left(\frac{A_\theta(\mid u\mid ^2)}{r}\right)\left(\frac{A_\theta(\mid u\mid ^2)}{r}\right) u \right\| _{L_{[t_1,\infty)}^{2 } L_{x}^{1}} + \left\| \left(\frac{A_\theta(\mid u\mid ^2)}{r}\right)^2 \nabla u \right\| _{L_{[t_1,\infty)}^{2 } L_{x}^{1}}  \\
& \quad + \left\| \left(\frac{A_\theta(\mid u\mid ^2)}{r}\right)^2  u \right\| _{L_{[t_1,\infty)}^{2 } L_{x}^{1}} \\
& \lesssim \left\| \frac{A_\theta(\mid u\mid ^2)^2}{r^3} u \right\| _{L_{[t_1,\infty)}^{2 } L_{x}^{1}}  + \left\|  \frac{A_\theta(\mid u\mid ^2)}{r} \mid u\mid ^3 \right\| _{L_{[t_1,\infty)}^{2 } L_{x}^{1}} \\
& \quad + \left\| \left(\frac{A_\theta(\mid u\mid ^2)}{r}\right)^2  \left(\nabla u +u \right) \right\| _{L_{[t_1,\infty)}^{2 } L_{x}^{1}}.
\end{align*}
Using \eqref{eq.esmk1} and  \eqref{eq.esmk2}, we can derive that
\begin{align} \label{e11}
\begin{split}
\left\| \frac{A_\theta(\mid u\mid ^2)^2}{r^3} u \right\| _{L_{[t_1,\infty)}^{2} L_{x}^{1}}&\lesssim \left\|  \left(\frac{A_\theta(\mid u\mid ^2)}{r^2}\right)^{\frac 3 2} \right\| _{L_{[t_1,\infty)}^{2} L_{x}^{2}} \left\|  \left(A_\theta(\mid u\mid ^2)\right)^{\frac 1 2} \right\| _{L_{[t_1,\infty)}^{\infty} L_{x}^{\infty}} \left\|  u \right\| _{L_{[t_1,\infty)}^{\infty} L_{x}^{2}} \\
&\lesssim\left\| \frac{A_\theta(\mid u\mid ^2)}{r^2} \right\| _{L_{[t_1,\infty)}^{3} L_{x}^{3}}^{\frac 3 2} \lesssim
\| u\| _{L^6_{[t_1,\infty)} L_x^6}^3.
\end{split}
\end{align}
and
\begin{align*}
\left\| \left(\frac{A_\theta(\mid u\mid ^2)}{r}\right)^2 \left(\nabla u +u\right) \right\| _{L_{[t_1,\infty)}^{2 } L_{x}^{1}} &\lesssim \left\|  \frac{A_\theta(\mid u\mid ^2)}{r^2} \right\| _{L_{[t_1,\infty)}^{2} L_{x}^{2}} \left\| A_\theta(\mid u\mid ^2)\right\| _{L_{[t_1,\infty)}^{\infty} L_{x}^{\infty}} \left\| u \right\| _{L_{[t_1,\infty)}^{\infty} H^1} \\
&\lesssim \| u\| _{L^4_{[t_1,\infty)} L_x^4}^2.
\end{align*}
In addition, we have that
$$
\left\|  \frac{A_\theta(\mid u\mid ^2)}{r} \mid u\mid ^3 \right\| _{L_{[t_1,\infty)}^{2 } L_{x}^{1}} \lesssim\left\| \frac{A_\theta(\mid u\mid ^2)}{r^2}  \right\| _{L_{[t_1,\infty)}^{2} L_{x}^{2}}\left\| r\mid u\mid ^2 \right\| _{L_{[t_1,\infty)}^{\infty} L_{x}^{\infty}}\left\| u\right\| _{L_{[t_1,\infty)}^{\infty} L_{x}^{2}} \lesssim \| u\| _{L^4_{[t_1,\infty)} L_x^4}^2.
$$
where we also used the Strauss inequality \eqref{Strauss}. From estimates above, then we conclude that
\begin{align*}
\left\| \Lambda_2(u) \right\| _{L^{2 }([t_1,\infty); W^{1,1}(\mathbb{R}^2))}  \lesssim \| u\| _{L^4_{[t_1,\infty)} L_x^4}^2
+ \| u\| _{L^6_{[t_1,\infty)} L_x^6}^3.
\end{align*}
We now turn to deal with the nonlocal $\Lambda_1$. Notice that
$$
\nabla \Lambda_1(u)=\nabla A_0(\mid u\mid ^2) u +  A_0(\mid u\mid ^2) \nabla u=\frac{A_\theta(\mid u\mid ^2)}{r} \mid u\mid ^3+ A_0(\mid u\mid ^2) \nabla u.
$$
This implies that
\begin{align*}
\left\| \Lambda_1(u) \right\| _{L^{2 }[t_1, \infty); W^{1,1}(\mathbb{R}^2))} \lesssim \left\|  \frac{A_\theta(\mid u\mid ^2)}{r} \mid u\mid ^3\right\| _{L_{[t_1, \infty)}^{2 } L_{x}^{1}} +
\left\| A_0(\mid u\mid ^2) \left(\nabla u +u\right) \right\| _{L_{[t_1, \infty)}^{2} L^1_x}.
\end{align*}
It follows from Lemma \ref{lem.bf1} that
$$
\left\| A_0(\mid u\mid ^2) \right\| _{L^2(\mathbb{R}^2)} \lesssim  \| u\| ^4_{L^{\frac8 3}(\mathbb{R}^2)}.
$$

$$
\left\| A_0(\mid u\mid ^2) \left(\nabla u +u\right) \right\| _{L_{[t_1, \infty)}^{2} L^1_x} \lesssim \left\| A_0(\mid u\mid ^2) \right\| _{L_{[t_1, \infty)}^{2} L^2_x} \left\| u\right\| _{L_{[t_1, \infty)}^{\infty}H^1} \lesssim \| u\| ^4_{L^8_{[t_1, \infty)} L_x^{\frac 8 3}}.
$$
This together with \eqref{e11} results in
$$
\left\| \Lambda_1(u) \right\| _{L^{2 }([t_1, \infty); W^{1,1}(\mathbb{R}^2))} \lesssim \| u\| _{L_{[t_1, \infty)}^4 L_x^4}^2 + \| u\| _{L_{[t_1, \infty)}^8 L_x^{\frac 8 3}}^4.
$$
This completes the proof.
\end{proof}

At this point, record that $u$ satisfies the following integral equation,
$$
u(t) = e^{\mathrm{i} (t-t_1) \Delta} u(t_1) +\mathrm{i} \int_{t_1}^t e^{\mathrm{i} (t-\tau)\Delta} \Lambda(u(\tau)) \,ds.
$$
Making use of  \eqref{eq.sm21m} and Proposition \ref{p.mca1} along with iterative arguments, then we can obtain \eqref{eq.key1m}.

\subsection{Estimate of Nemytskii operator} We shall establish an estimate for the following Nemytskii operator,
\begin{equation*} 
u (x) \mapsto f(u)(x)=u(x)\mid u(x)\mid ^{p-1}, \quad  x \in \mathbb{R}^n.
\end{equation*}
For this purpose, we need to introduce
the Besov space ${B^\sigma_{q,2}(\mathbb{R}^n)}$ equipped with the norm
\begin{equation*} 
\begin{aligned}
&\| u\| _{B^\sigma_{q,2}(\mathbb{R}^n)} =
 \| u\| _{L^q(\mathbb{R}^n)} +  \left\| \ \left\|  \frac{\mid u(x)-u(x-h)\mid }{\mid h\mid ^{\frac n 2+\sigma}}  \right\| _{L^{q}(R^n_x)} \right\| _{L^{2}(\mathbb{R}^n_h)}, \quad u \in B^{\sigma}_{q,2}(\mathbb{R}^n).
\end{aligned}
\end{equation*}
\begin{lem} \label{l.Nem1}
Let $0<s<1$, $1<r<2$ and $n \geq 2$, then, for any $s <\sigma <1$,  there holds that
\begin{align*}
\| D^s f(u)\| _{L^r(\mathbb{R}^n)} \lesssim \| u\| _{B^\sigma_{q_1,2}(\mathbb{R}^n)} \| u\| _{L^{q_2(p-1)}(\mathbb{R}^n)}^{p-1},
\end{align*}
where
$$
\frac{1}{r} = \frac{1}{q_1} + \frac{1}{q_2}, \quad q_1 \geq 2.
$$
\end{lem}
\begin{proof}
By the definition of the fractional Laplacian, we first write
\begin{align*}
D^s f(u)(x) = c_{n, s} \int_{\mathbb{R}^n} \frac{f(u)(x)-f(u)(y)}{\mid x-y\mid ^{n+s}} \, dy,
\end{align*}
where $c_{n,s}>0$ is a constant given by
$$
c_{n,s}=\frac{2^s \Gamma(\frac n 2 +\frac s 2)}{\pi^{\frac n 2} \mid \Gamma(-\frac s 2)\mid }.
$$
Observe that
$$
\mid f(u)(x)-f(u)(y)\mid  \lesssim \mid u(x)-u(y)\mid  \left( \mid u(x)\mid ^{p-1}+ \mid u(y)\mid ^{p-1} \right).
$$
Then we have that
\begin{align*}
\| D^s f(u)\| _{L^r(\mathbb{R}^n)}^r & \lesssim \underbrace{ \int_{\mathbb{R}^n} \left(\int_{\mathbb{R}^n} \frac{\mid u(x)-u(y)\mid \mid u(x)\mid ^{p-1}}{\mid x-y\mid ^{n+s}} \,dy\right)^r \,dx}_{I} \\
& \quad + \underbrace{\int_{\mathbb{R}^n} \left(\int_{\mathbb{R}^n} \frac{\mid u(x)-u(y)\mid \mid u(y)\mid ^{p-1}}{\mid x-y\mid ^{n+s}} \,dy\right)^r \,dx}_{II}.
\end{align*}
It is straightforward to verify that
$$
I \lesssim I_+ + I_-, \quad  II \lesssim II_+ + II_-,
$$
where
\begin{align*}
I_+&=\int_{\mathbb{R}^n} \left(\int_{\mid x-y\mid  > 1} \frac{\mid u(x)-u(y)\mid \mid u(x)\mid ^{p-1}}{\mid x-y\mid ^{n+s}} \,dy\right)^r \,dx, \\
I_-&=\int_{\mathbb{R}^n} \left(\int_{\mid x-y\mid  \leq 1} \frac{\mid u(x)-u(y)\mid \mid u(x)\mid ^{p-1}}{\mid x-y\mid ^{n+s}} \, dy\right)^r \, dx
\end{align*}
and
\begin{align*}
II_+&=\int_{\mathbb{R}^n} \left(\int_{\mid x-y\mid  > 1} \frac{\mid u(x)-u(y)\mid \mid u(y)\mid ^{p-1}}{\mid x-y\mid ^{n+s}} \, dy\right)^r \, dx,\\
II_-&=\int_{\mathbb{R}^n} \left(\int_{\mid x-y\mid  \leq 1} \frac{\mid u(x)-u(y)\mid \mid u(y)\mid ^{p-1}}{\mid x-y\mid ^{n+s}} \, dy\right)^r \, dx.
\end{align*}
In what follows, we are going to estimate every term above. Let us first treat the term $I_+$. From H\"older's inequality, we see that
\begin{align*}
\int_{\mid x-y\mid > 1} \frac{\mid u(x)-u(y)\mid }{\mid x-y\mid ^{n+s}} dy &\leq  \left(\int_{\mid x-y\mid  > 1} \frac{\mid u(x)-u(y)\mid ^2}{\mid x-y\mid ^{n+2(s-\delta)}} \, dy \right)^{\frac 1 2}\left(\int_{\mid x-y\mid  > 1} \frac{1}{\mid x-y\mid ^{n+2\delta}} \, dy \right)^{\frac 1 2} \\
&\lesssim \left(\int_{\mid x-y\mid  > 1} \frac{\mid u(x)-u(y)\mid ^2}{\mid x-y\mid ^{n+2(s-\delta)}} \, dy \right)^{\frac 1 2},
\end{align*}
where $0<\delta<s$. Therefore, we get that
\begin{align*}
I_+ \lesssim \int_{\mathbb{R}^n}  \left(\int_{\mid x-y\mid  > 1} \frac{\mid u(x)-u(y)\mid ^2}{\mid x-y\mid ^{n+2(s-\delta)}} \, dy \right)^{\frac r 2} \mid u(x)\mid ^{r(p-1)} \, dx.
\end{align*}
Invoking H\"older's inequality again, we have that
\begin{align}\label{i1}
I_+ \lesssim \left(\int_{\mathbb{R}^n}  \left(\int_{\mid x-y\mid  > 1} \frac{\mid u(x)-u(y)\mid ^2}{\mid x-y\mid ^{n+2(s-\delta)}} \, dy \right)^{\frac{q_1}{2}}\,dx\right)^{\frac{r}{q_1}} \left(\int_{\mathbb{R}^n} \mid u(x)\mid ^{q_2(p-1)} \,dx\right)^{\frac{r}{q_2}},
\end{align}
where
$$
\frac{1}{r}= \frac{1}{q_1} + \frac{1}{q_2}, \quad q_2 = \frac{q_1r}{q_1-r}.
$$
Notice that
\begin{align} \label{i11} \nonumber
& \left(\int_{\mathbb{R}^n}  \left(\int_{\mid x-y\mid  > 1} \frac{\mid u(x)-u(y)\mid ^2}{\mid x-y\mid ^{n+2(s-\delta)}} \,dy \right)^{\frac{q_1}{2}}dx\right)^{\frac{1}{q_1}}  \leq \\  \nonumber&\leq  \left\| \ \left\|  \frac{\mid u(x)-u(y)\mid }{\mid x-y\mid ^{\frac n 2+(s-\delta)}}  \right\| _{L^2(R^n_y)} \right\| _{L^{q_1}(\mathbb{R}^n_x)} \\
&= \left\| \ \left\|  \frac{\mid u(x)-u(x-h)\mid }{\mid h\mid ^{\frac n 2+(s-\delta)}}  \right\| _{L^2(R^n_h)} \right\| _{L^{q_1}(\mathbb{R}^n_x)}\\ \nonumber
& \leq \left\| \ \left\|  \frac{\mid u(x)-u(x-h)\mid }{\mid h\mid ^{\frac n 2+(s-\delta)}}  \right\| _{L^{q_1}(R^n_x)} \right\| _{L^{2}(\mathbb{R}^n_h)}  \leq \| u\| _{B^{s-\delta}_{q_1,2}(\mathbb{R}^n)}.
\end{align}
This along with \eqref{i1} leads to
\begin{align}\label{eq.an8}
I_+ \lesssim \left( \| u\| _{B^{s-\delta}_{q_1,2}(\mathbb{R}^n)} \| u\| _{L^{q_2(p-1)}}^{p-1}\right)^r.
\end{align}
By a similar way and choosing $\delta=\sigma-s>0$, we can derive that
\begin{align}\label{eq.an9}
I_- \lesssim \left( \| u\| _{B^{\sigma}_{q_1,2}(\mathbb{R}^n)} \| u\| _{L^{q_2(p-1)}}^{p-1}\right)^r.
\end{align}
Next we deal with the term $II_+$. Applying H\"older's inequality, we first have that
\begin{align*}
\int_{\mid x-y\mid > 1} \frac{\mid u(x)-u(y)\mid \mid u(y)\mid ^{p-1}}{\mid x-y\mid ^{n+s}} dy \leq  \left(\int_{\mid x-y\mid  > 1} \frac{\mid u(x)-u(y)\mid ^2}{\mid x-y\mid ^{n+2(s-\delta)}}\, dy \right)^{\frac 12}\left(\int_{\mid x-y\mid  > 1} \frac{\mid u(y)\mid ^{2(p-1)}}{\mid x-y\mid ^{n+2\delta}} \, dy \right)^{\frac 1 2}.
\end{align*}
This results in
\begin{align}\label{i2}
\begin{split}
II_+  &\leq \left\| \left(\int_{\mid x-y\mid  > 1} \frac{\mid u(x)-u(y)\mid ^2}{\mid x-y\mid ^{n+2(s-\delta)}} \, dy \right)^{\frac 1 2}\left(\int_{\mid x-y\mid  > 1} \frac{\mid u(y)\mid ^{2(p-1)}}{\mid x-y\mid ^{n+2\delta}} \, dy \right)^{\frac 1 2} \right\| _{L^r(\mathbb{R}^n)}^r \\
& \leq \left\| \left(\int_{\mid x-y\mid  > 1} \frac{\mid u(x)-u(y)\mid ^2}{\mid x-y\mid ^{n+2(s-\delta)}} dy \right)^{\frac 1 2} \right\| _{L^{q_1}(\mathbb{R}^n)}^r \left\| \left(\int_{\mid x-y\mid  > 1} \frac{\mid u(y)\mid ^{2(p-1)}}{\mid x-y\mid ^{n+2\delta}} \,dy \right)^{\frac 1 2}    \right\| _{L^{q_2}(\mathbb{R}^n)}^r \\
&=\left(\int_{\mathbb{R}^n}\left(\int_{\mathbb{R}^n} \frac{\mid u(x)-u(y)\mid ^2}{\mid x-y\mid ^{n+2(s-\delta)}} \, dy\right)^{\frac{q_1}{2}} \,dx \right)^{\frac{r}{q_1}}
\left\| \int_{\mid x-y\mid  > 1} \frac{\mid u(y)\mid ^{2(p-1)}}{\mid x-y\mid ^{n+2\delta}} \,dy\right\| _{L^{\frac{q_2}{2}}(\mathbb{R}^n)}^{\frac r 2},
\end{split}
\end{align}
where we also used H\"older's inequality. In view of Young's inequality, we know that
\begin{align*}
\left\| \int_{\mid x-y\mid  > 1} \frac{\mid u(y)\mid ^{2(p-1)}}{\mid x-y\mid ^{n+2\delta}} \,dy \right\| _{L^{\frac{q_2}{2}}(\mathbb{R}^n)}
&=\left\| \int_{\mathbb{R}^n} \frac{\mathds{1}_{\mid x-y\mid >1}}{\mid x-y\mid ^{n+2\delta}} \mid u(y)\mid ^{2(p-1)}\,dy \right\| _{L^{\frac{q_2}{2}}(\mathbb{R}^n)} \\
& \leq \left\| \frac{\mathds{1}_{\mid \cdot\mid >1}}{\mid \cdot\mid ^{n+2\delta}} \right\| _{L^1(\mathbb{R}^n)} \| \mid u\mid ^{2(p-1)}\| _{L^{\frac{q_2}{2}}(\mathbb{R}^n)} \lesssim \| u(y)\| _{L^{q_2(p-1)}(\mathbb{R}^n)}^{2(p-1)}.
\end{align*}
It then follows from \eqref{i11} and \eqref{i2} that
\begin{align}\label{eq.an10}
II_+ \lesssim \left( \| u\| _{B^{s-\delta}_{q_1,2}(\mathbb{R}^n)} \| u\| _{L^{q_2(p-1)}}^{p-1}\right)^r
\end{align}
Similarly, we can obatin that
\begin{align}\label{eq.an11}
II_- \lesssim \left( \| u\| _{B^{\sigma}_{q_1,2}(\mathbb{R}^n)} \| u\| _{L^{q_2(p-1)}}^{p-1}\right)^r.
\end{align}
Combining \eqref{eq.an8}, \eqref{eq.an9}, \eqref{eq.an10} and \eqref{eq.an11}, then we have the desired conclusion. This completes the proof.
\end{proof}

Taking advantage of the well-known inclusions between Besov and Sobolev spaces, we now get the following result.

\begin{lem} \label{l.Nem2}
Let $0<s<1$, $1<r<2$ and $n \geq 2$, then, for any $s <\sigma <1$, there holds that
\begin{align*}
\| f(u)\| _{W^{s,r}(\mathbb{R}^n)} \lesssim \| u\| _{W^{\sigma,q_1}(\mathbb{R}^n)} \| u\| _{L^{q_2(p-1)}(\mathbb{R}^n)}^{p-1},
\end{align*}
where
$$
\frac{1}{r} = \frac{1}{q_1} + \frac{1}{q_2}, \quad q_1 \geq 2.
$$
\end{lem}

\subsection*{Acknowledgements}

The first  author was supported in part by  INDAM, GNAMPA-Gruppo Nazionale per l'Analisi Matematica, la Probabilita e le loro Applicazioni, by Institute of Mathematics and Informatics, Bulgarian Academy of Sciences, by Top Global University Project, Waseda University and  the Project PRA 2018 49 of  University of Pisa and by the project PRIN  2020XB3EFL funded by the Italian Ministry of Universities and Research. The second author was supported by the Postdoctoral
Science Foundation of China (No.Y990071G21). Part of the work has been completed while the second author visited the Department of Mathematics, University of Pisa, whose hospitality he gratefully acknowledges.


\begin{thebibliography}{10}

\bibitem{ADM}  {A. K. Arora, B. Dodson and J. Murphy,} \textit{Scattering below the ground state for the 2D radial nonlinear Schr\"odinger equation,} Proc. Amer. Math. Soc. 148 (2020), no. 4, 1653--1663.


\bibitem{BBS} L. Berg\'e, A. de Bouard and J.-C. Saut, \textit{Blowing up time-dependent solutions of the planar, Chern-Simons gauged nonlinear Schr\"odinger equation}, Nonlinearity 8 (1995), no. 2, 235--253.

\bibitem{BL} {T. Boulenger and E. Lenzmann:} \textit{Blowup for biharmonic NLS,} Ann. Sci. \'Ec. Norm. Sup\'er (4)  50 (2017), no. 3, 503--544.

\bibitem{Ca} T. Cazenave, \textit{Semilinear Schr\"odinger Equations}, Courant Lecture Notes in Mathematics, 10. Amer. Math. Soc., Providence, RI, 2003.


 \bibitem{CK01}   M. Christ and A. Kiselev, Maximal functions associated to filtrations, Journ. Func. Anal. 179
(2001), 409 -- 425.

\bibitem{Du} G. Dunne, \textit{Self-Dual Chern-Simons Theories}, Lecture Notes in Physics, Springer, 1995.

\bibitem{EHA} Z. F. Ezawa, M. Hotta and A. Iwazaki, \textit{Breathing vortex solitons in nonrelativistic Chern-Simons gauge theory}, Phys. Rev. Lett. 67 (1991), no. 4, 411--414.

\bibitem{EHA1} Z. F. Ezawa, M. Hotta and A. Iwazaki, \textit{Nonrelativistic Chern-Simons vortex solitons in external magnetic field}, Phys. Rev. D 44 (1991), no. 2, 452--463.



\bibitem{Gou} T. Gou and Z. Zhang, \textit {Normalized solutions to the Chern-Simons-Schr\"odinger system,}  J. Funct. Anal. 280 (2021), no. 5, 108894, 65 pp.

\bibitem{HoZh} P. Horvathy and P. Zhang, \textit{Vorticesin (Abelian) Chern-Simons gauge theory}, Phys. Rep.,  481 (2009), no. 5-6, 83--142.

\bibitem{Hu1} {H. Huh,} \textit{Blow-up solutions of the Chern-Simons-Schr\"odinger equations,} Nonlinearity 22 (2009), no. 5, 967--974.

\bibitem{Huh} H. Huh, \textit{Energy solution to the Chern-Simons-Schr\"odinger equations}, Abstr. Appl. Anal., 2013, Article ID 590653.

\bibitem{JaPi1} R. Jackiw and S.-Y. Pi, \textit{Classical and quantal nonrelativistic Chern-Simons theory,} Phys. Rev. D, 42 (1990), no. 10, 3500--3513.

\bibitem{JaPi2} R. Jackiw and S.-Y. Pi, \textit{Self-dual Chern-Simons solitons,} Progress of Theoretical Physics. Supplement, (1992), no. 107, 1--40.

\bibitem{KeTao98} M. Keel and T. Tao, \textit{Endpoint Strichartz estimates,} Amer. J. Math., 120 (1998), 955--980.

\bibitem{KM1} {C. Kenig and F. Merle,}  \textit{Global well-posedness, scattering and blow-up for the energy
critical, focusing, non-linear Schr\"odinger equation in the radial case,} Invent. Math. 166 (2006), 645--675.

\bibitem{KM2} {C. Kenig and F. Merle,} \textit{Global well-posedness, scattering and blow-up for the energy critical, focusing, non-linear wave equation,} Acta Math. 201 (2008), 147--212.

\bibitem{KKO} K. Kim, S. Kwon and S.-J. Oh, \textit{Blow-up dynamics for smooth finite energy radial data solutions to the self-dual Chern-Simons-Schr\"odinger equation}, arXiv:2010.03252.

\bibitem{LL} Z. Li and B. Liu, \textit{On Threshold Solutions of equivariant Chern-Simons-Schr\"odinger Equation}, arXiv:2010.09045.


\bibitem{Lieb} E.H. Lieb and M. Loss, \textit{Analysis}, American Mathematical Society, 2001.

\bibitem{Lim} Z. M. Lim, \textit{Large data well-posedness in the energy space of the Chern-Simons-Schr\"odinger system,} J. Differential Equations, 264 (2018), no. 4,  2553--2597.

\bibitem{LS16} B. Liu and P. Smith, \textit{Global wellposedness of the equivariant Chern-Simons-Schr\"odinger equation}, Rev. Mat. Iberoam., 32 (2016), no. 3, 751--794.


\bibitem{M72} B. Muckenhoupt, \textit{Hardy's inequality with weights,} Studia Math. 44 (1972), 31-38.

\bibitem{OhPu} S.-J. Oh and F.  Pusateri, \textit{Decay and scattering for the Chern-Simons-Schr\"odinger equations,} Int. Math. Res. Not. IMRN, 2015 (2015), no. 24, 13122--13147.

\bibitem{S01} A. Stefanov,  \textit{Strichartz estimates for the Schr\"odinger equation with radial data},
Proc. Amer. Math. Soc., 129 (2001), no. 5, 1395--1401.

\bibitem{T00} T. Tao, \textit{Spherically averaged endpoint Strichartz estimates for the two-dimensional
Schr\"odinger equation,} Comm. Partial Differential Equations, 25 (2000), no. 7 - 8,
1471--1485.

\bibitem{Ya} Y. Yang, \textit{ Solitons in Field Theory and Nonlinear Analysis}, Springer Monographs in Mathematics, New York, Springer, 2001.

\end{thebibliography}
\end{document}